\documentclass[a4paper]{amsart}
\usepackage[utf8]{inputenc}
\usepackage{amsmath,amssymb,amsthm}
\usepackage{hyperref,graphicx,color}
\usepackage[shortlabels]{enumitem}

\newtheorem{thm}{Theorem}
\newtheorem{lem}[thm]{Lemma}
\newtheorem{prop}[thm]{Proposition}
\newtheorem{cor}[thm]{Corollary}
\newtheorem{rem}[thm]{Remark}
\newtheorem{conj}[thm]{Conjecture}

\numberwithin{thm}{section}

\newtheorem{dfn}{Definition}
\numberwithin{dfn}{section}

\newtheorem{assum}{Assumption}

\title{A reverse Faber-Krahn inequality for the magnetic Laplacian}

\author{Bruno Colbois}
\author{Corentin Léna}
\author{Luigi Provenzano}
\author{Alessandro Savo}

\address{Bruno Colbois, Université de Neuch\^atel, Institute de Mathémathiques, Rue Emile Argand 11, 2000 Neuch\^atel, Switzerland}
\email{bruno.colbois@unine.ch}
\address{Corentin Léna, Università degli Studi di Padova, Dipartimento di Tecnica e Gestione dei Sistemi Industriali (DTG), Stradella S. Nicola 3, 36100 Vicenza, Italy\\
\newline
\indent Università degli Studi di Padova, Dipartimento di Matematica ``Tullio Levi-Civita'', via Trieste 63, 35121 Padova, Italy}
\email{corentin.lena@unipd.it}
\address{Luigi Provenzano, Sapienza Università  di Roma, Dipartimento di Scienze di Base e Applicate per l'Ingegneria (SBAI), Via Antonio Scarpa 16, 00161 Roma, Italy}
\email{luigi.provenzano@uniroma1.it}
\address{Alessandro Savo, Sapienza Università  di Roma, Dipartimento di Scienze di Base e Applicate per l'Ingegneria (SBAI), Via Antonio Scarpa 16, 00161 Roma, Italy}
\email{alessandro.savo@uniroma1.it}

\thanks{The first author acknowledges support of the SNSF project ‘‘\emph{Geometric Spectral Theory}’’, grant number 200020\_212570. The second author acknowledges support from the INdAM GNAMPA Project  ‘‘\emph{Operatori differenziali e integrali in geometria spettrale}’’ (CUP E53C22001930001). The third author acknowledges support of the project "Perturbation problems and asymptotics for elliptic differential equations: variational and potential theoretic methods" funded by the European Union -- Next Generation EU and by MUR-PRIN-2022SENJZ3. The third and fourth author acknowledge support from the INdAM GNSAGA Project ‘‘\emph{Analisi Geometrica: Equazioni alle Derivate Parziali e Teoria delle Sottovarietà}’’ (CUP E53C22001930001). The authors wish to thank the anonymous referees for suggesting several corrections and improvements.}



\date{September 14, 2024}

\begin{document}

\begin{abstract}
We consider the first eigenvalue of the magnetic Laplacian in a bounded and simply connected planar domain, with uniform magnetic field and Neumann boundary conditions. We investigate the reverse Faber-Krahn inequality conjectured by S.~Fournais and B.~Helffer, stating that this eigenvalue is maximized by the disk for a given area. Using the method of level lines, we prove the conjecture for small enough values of the magnetic field (those for which the corresponding eigenfunction in the disk is radial). 
\end{abstract}

\maketitle

\section{Introduction}

\subsection{Magnetic Laplacians}
\label{subsecMagnLap}

The present work investigates a conjectured \emph{reverse Faber-Krahn inequality} for the Laplacian with constant magnetic field (Conjecture \ref{conjRFK}), and establishes it for a small enough value of the field (Theorem \ref{thmRFK}).  In order to fully understand the problem, let us start from a more general point of view. We consider a vector potential, that is, a real-valued vector field $A:\mathbb R^2\to\mathbb R^2$ of class $C^\infty$, written
\begin{equation*}
	A(p):=	\left(\begin{array}{c}
					A_1(p)\\
					A_2(p)
			\end{array}
	\right)
\end{equation*} 
(where $p=(x,y)$ denotes a point in $\mathbb R^2$).
We associate with $A$ a magnetic gradient
\begin{equation*}
  \nabla^{A}u:=\nabla u-i\,u\,A,
\end{equation*}
a magnetic divergence 
\begin{equation*} 
	\mbox{div}_{A} X:=\mbox{div}\, X-i\,A\cdot X 
\end{equation*}
and a magnetic Laplacian
\begin{equation}\label{eqMagLap} 
	\Delta_{A}u:=-\mbox{div}_{A} \left(\nabla^{A}u\right)=\Delta u +2i\,A\cdot\nabla  u+(|A|^2+i\,\mbox{div}(A))u.
\end{equation}
Let us stress that we are using the following definition, common in differential geometry, for  the Laplace operator:
\[\Delta u:=-{\rm div}\,\nabla u=-\partial_{xx}^2u-\partial_{yy}^2u.\]
For all the vector potentials $A$ considered in the paper, $\mbox{div}(A)=0$. We define the magnetic field associated with $A$  as the function 
\begin{equation*}
	\beta(p):=\mbox{curl}\,A(p)= \partial_x A_2(p)-\partial_yA_1(p).
\end{equation*}

The differential operator $\Delta_{A}$ is formally symmetric and uniformly elliptic, with smooth coefficients. Given an open and bounded set $\Omega\subset\mathbb R^2$ with a sufficiently smooth boundary (we assume that $\partial \Omega$ is $C^\infty$ smooth throughout the paper),  we can consider several self-adjoint realizations of $\Delta_{A}$ as operators on $L^2(\Omega)$ and study their (discrete) spectrum. We are mainly interested in the Neumann realization $\Delta_{A}^N$, whose spectrum $(\lambda_k^N(\Omega,A))_{k\ge 1}$ consists of the eigenvalues (counted with repetitions) of the problem 
\begin{equation}\label{eqEVNeuman}
\left\{\begin{array}{rll}
		\Delta_{A} u&=\lambda u&\mbox{ in }\Omega,\\
		\nabla^{A}u \cdot \nu&=0&\mbox{ on }\partial\Omega,
		\end{array}
\right.
\end{equation} 
where $\nu$ denotes the outward pointing normal vector. We have $\lambda_k^N(\Omega,A)\ge 0$ for all $k $ and $\lambda_k^N(\Omega,A)\to +\infty$ as $k\to\infty$. Another operator of interest is the Dirichlet realization, denoted by $\Delta^D_A$, whose spectrum $(\lambda_k^D(\Omega,A))_{k\ge1}$ consists of the eigenvalues of the problem 
\begin{equation}
\label{eqEVDirichlet}
\left\{\begin{array}{rll}
		\Delta_{A} u&=\lambda u&\mbox{ in }\Omega,\\
		u&=0&\mbox{ on }\partial\Omega,
		\end{array}
\right.
\end{equation}
which are also non-negative and tend to $+\infty$. Let us note that $\Delta_A^N$ and $\Delta_A^D$ can be constructed using a vector potential only defined on $\overline{\Omega}$, that is to say assuming only $A\in C^\infty(\overline{\Omega},\mathbb R^2)$.

For completeness, let us give a rigorous definition of $\Delta_A^N$. We define a closed and non-negative quadratic form, with domain $H^1(\Omega)$, by
\[a(u):=\int_\Omega \left|\nabla^A u\right|^2\,dx.\]
According to standard results in spectral theory, there exists a unique self-adjoint operator $T$ whose domain $\mathcal D(T)$ is contained in $H^1(\Omega)$ and such that, for all $u\in\mathcal D(T)$ and $v\in H^1(\Omega)$, 
\[\langle Tu,v\rangle=b(u,v),\]
where $\langle\cdot,\cdot\rangle$ denotes the scalar product in $L^2(\Omega)$ and $b(\cdot,\cdot)$ the sesquilinear form associated with $a(\cdot)$. We then define $\Delta_\Omega^N$ to be this operator $T$. We proceed in the same way, starting with the form domain $H^1_0(\Omega)$, to define $\Delta_A^D$.

These operators enjoy a so-called \emph{gauge invariance} property, which can be formulated as follows. Given a smooth function $f:\overline{\Omega}\to \mathbb R$, we define the modified vector potential $\hat A:=A+\nabla f$ and consider the associated magnetic Laplacian $\Delta_{\hat A}$. We have $\Delta_{\hat A}(e^{i f}u)=e^{i f}\Delta_{A}u$ for every smooth function $u$. This property of the differential operator implies that the self-adjoint operator $\Delta_{\hat A}^N$ is isospectral to $\Delta_{A}^N$, so that $\lambda^N_k(\Omega,\hat A)=\lambda_k^N(\Omega,A)$ for every $k$. A similar result holds for the Dirichlet realization.

If $\Omega$ is simply connected, the above gauge-invariance property implies that the spectrum $(\lambda^N_k(\Omega,A))_{k\ge1}$ only depends on the magnetic field $\beta(p)$. Indeed, let us assume that $\tilde A$ is another vector potential such that, for all $p\in \Omega$, 
\[\mbox{curl}\,\tilde A(p)= \beta(p)=\mbox{curl}\, A(p).\]
Then, according to Poincaré's lemma, there exists a smooth function $f:\overline{\Omega}\to\mathbb R$ such that $\tilde A=A+\nabla f$. It then follows from gauge invariance that $\lambda^N_k(\Omega,\tilde A)=\lambda_k^N(\Omega,A)$ for all $k$, as claimed. The same holds true for the Dirichlet eigenvalues.

\subsection{The case of a constant magnetic field}

We now restrict our attention to particular classes of vector potentials. We first consider
\[A=\beta\,A_S,\]
where $\beta\in\mathbb R$ is a constant and $A_S$ denotes the standard vector potential
 \begin{equation*}
	A_S(x,y):=	\frac12\left(\begin{array}{c}
					-y\\
					x
			\end{array}
	\right).
\end{equation*}
We have, for all $p\in\mathbb R^2$,
\[\mbox{curl}\,(\beta A_S)(p)=\beta,\]
meaning that the magnetic field associated with $\beta\,A_S$ has the constant value $\beta$. We write $(\lambda_k(\Omega,\beta))_{k\ge1}$ for the eigenvalues of $\Delta_{\beta A_S}^N$ (that is, the eigenvalues of Problem \eqref{eqEVNeuman} with $A=\beta A_S$). Similarly, we write 
$\lambda^D_k(\Omega,\beta)$ for $\lambda^D_k(\Omega,\beta A_S)$. As seen immediately from Formula \eqref{eqMagLap}, for any smooth complex-valued function $u$, $\Delta_{-\beta A_S}\overline{u} =\overline{\Delta_{\beta A_S}u}$, where the bar denotes complex conjugation. This implies $\lambda_k^N(\Omega,-\beta)=\lambda_k^N(\Omega,\beta)$ for all $k$ and $\beta$. We can therefore assume, without loss of generality, that $\beta\ge0$. 

Next, we consider vector potentials constructed using the so-called \emph{torsion function} $\psi$ of the domain $\Omega$. Among several competing conventions, we choose the following: we define $\psi$ as the unique solution of the boundary value problem
\begin{equation}\label{eqBVPTorsion}
\left\{\begin{array}{rcl}
		\Delta\psi&=1&\mbox{ in }\Omega,\\
		\psi&=0&\mbox{ on }\partial\Omega.
		\end{array}
\right.
\end{equation} 
We then define the vector potential $A_\Omega:\overline\Omega\to\mathbb{R}^2$ by 
\begin{equation}\label{eqVectPot}
	A_\Omega(p):=-\nabla^\perp\psi(p),
\end{equation}
or more explicitly
\begin{equation*}
	\left(\begin{array}{c}
					A_{\Omega,1}(p)\\
					A_{\Omega,2}(p)
			\end{array}
	\right):=
	\left(\begin{array}{c}
					\partial_y\psi(p)\\
					-\partial_x\psi(p)
			\end{array}
	\right).
\end{equation*}
By construction, we have
\begin{equation}\label{eqCurl}
	\mbox{curl }A_\Omega(p):=\partial_x A_{\Omega,2}(p)-\partial_y A_{\Omega,1}(p)=\Delta \psi=1.
\end{equation}
In the case where $\Omega$ is simply connected, it follows from our previous discussion of gauge-invariance that $\lambda^N_k(\Omega,\beta A_\Omega)=\lambda^N_k(\Omega,\beta)$ for all $k$ and $\beta$.

For a complete exposition of the material in this 
section and the previous one, we refer the interested reader to the book \cite{FournaisHelffer}, by S. Fournais and B. Helffer. These authors also develop the asymptotic analysis for large $\beta$ and apply it to the study of the Ginsburg-Landau functional in superconductivity.

\subsection{Constant magnetic field in a disk}\label{subsecDisk}

The special case where $\Omega$ is $B_R$, the disk of radius $R>0$ centered at $0$ in $\mathbb R^2$, and where $A=\beta A_S$ (with $\beta\ge0$), will be particularly useful. Let us briefly review the relevant results, referring to \cite[Appendix B]{CLPS2023Magnetic} for more detail. We can solve the eigenvalue problem \eqref{eqEVNeuman} by using the polar coordinates $(x,y)=(r\cos(\theta),r\sin(\theta))$ and looking for eigenfunctions of the form $v(r)e^{in\theta}$, with $n\in \mathbb Z$. We find that the spectrum of $\Delta_{\beta A_S}^N$ is the reunion, counted with multiplicities, of the spectra $(\kappa_k(n,\beta,R))_{k\ge1}$ of the following family of Sturm-Liouville problems, indexed by $n$:
\begin{equation}\label{eqSLPolar}
	\left\{\begin{array}{rll}
		-v''-\frac1rv'+\left(\frac{\beta r}{2}-\frac{n}{r}\right)^2v=\kappa v&\mbox{ in }]0,R[,\\
		\lim_{r\to0}r\,v'(r)=0,\\
		v'(R)=0.& 
		\end{array}
\right.
\end{equation}
In particular, for every $R>0$ and $\beta\ge0$,
\begin{equation}
	\lambda_1^N(B_R,\beta)=\inf_{n\in\mathbb Z}\kappa_1(n,\beta,R).
\end{equation}

To state and prove our main theorem, we need more precise information in the case of a small magnetic field. We consider the unit disk $B_1$. The following result is proved in Appendix \ref{appBetaStar}. It has also been obtained by A.~Kachmar and V.~Lotoreichik (see \cite[Proposition 3.3]{KL2022Magnetic}). 

\begin{prop}\label{propBetaStar} 
 For all $\beta\in[0,1]$,
 \begin{equation*}
 	\lambda_1^N(B_1,\beta)=\kappa_1(0,\beta,1).
 \end{equation*}
\end{prop}

\begin{dfn}\label{dfnBetaStar} We define the constant $\beta^*$ as the largest $\beta_0>0$ such that 
\[\lambda_1^N(B_1,\beta)=\kappa_1(0,\beta,1)\mbox{ for all }\beta\in[0,\beta_0].\]
Let us note that, according to Proposition \ref{propBetaStar}, $\beta^*$ exists and $\beta^*\ge1$.
\end{dfn}

\begin{rem}\label{remBetaStar}
Problem \eqref{eqSLPolar} can be studied by expressing solutions of the differential equation using generalized Laguerre functions and by computing the eigenvalues using numerical root finding (see \cite[Appendix B]{CLPS2023Magnetic}). It appears that 
\begin{enumerate}
	\item $\beta^*\approx 3.84754$;
	\item $\lambda_1^N(B_1,\beta)<\kappa_1(0,\beta,1)$ for all $\beta>\beta^*$.
\end{enumerate}
This is consistent with the earlier results of D. Saint-James \cite{SJ1965}, who performed a similar analysis, motivated by the study of superconductivity. It would be interesting to have precise numerical bounds for $\beta^*$ and a rigorous
proof of point (2).
\end{rem}

From the scaling properties of magnetic eigenvalues, we have 
\[\lambda_1^N(B_R,\beta)=R^{-2}\lambda_1^N(B_1,R^2\beta).\]
It then follows from Definition \ref{dfnBetaStar} that $\lambda_1^N(B_R,\beta)$ has an associated eigenfunction which is radial and real-valued if $R^2\beta\in[0,\beta^*]$.

Finally, let us point out that the eigenvalue problem \eqref{eqSLPolar} can be given an alternative form using the change of variable $a=\pi r^2$, with $f(a)=v(r)$. In the case $n=0$, we obtain
\begin{equation}\label{eqSLArea}
	\left\{\begin{array}{rll}
		-(4\pi\,a f')'+\frac{\beta^2 a}{4\pi}f=\kappa f&\mbox{ in }]0,a^*[,\\
		\lim_{a\to0}a\,f'(a)=0,&\\
		f'(a^*)=0,& 
		\end{array}
\right.
\end{equation}
where $a^*=\pi R^2=|B_R|$. According to the definition of $\beta^*$, the first eigenvalue $\kappa_1(0,\beta,R)$ of Problem \eqref{eqSLArea} coincides with $\lambda_1^N(B_R,\beta)$ if $\beta a^*\le \beta^*\pi$.

\subsection{Review of eigenvalue bounds}

In the previous paper \cite{CLPS2023Magnetic}, we provided upper and lower bounds for the magnetic Neumann eigenvalues $\lambda_k^N(\Omega,\beta)$, involving geometric parameters of $\Omega$. Although we devoted particular attention to $\lambda_1^N(\Omega,\beta)$, many of the results in \cite{CLPS2023Magnetic} apply to higher eigenvalues. In the present work, we are only concerned with comparing the first magnetic Neumann eigenvalue on a domain $\Omega$ with the corresponding one on the disk having the same area. Throughout the paper, we denote by $\Omega^*$ the disk in $\mathbb R^2$ centered at the origin such that $|\Omega^*|=|\Omega|$.

In the Dirichlet case (Problem \eqref{eqEVDirichlet}), L. Erd{\H o}s extended the classical Rayleigh-Faber-Krahn inequality to the magnetic Laplacian \cite{Erdos1996Inequality}.
\begin{thm}[Erd{\H o}s, 1996] \label{thmDirichlet} Let $\Omega\subset\mathbb R^2$ be an open and bounded set with a $C^\infty$ smooth boundary. Then, for all $\beta\ge0$,
\begin{equation*}
	\lambda_1^D(\Omega,\beta)\ge\lambda_1^D(\Omega^*,\beta).
\end{equation*}
Equality occurs if, and only if, $\Omega$ is a disk.
\end{thm}

The two-dimensional Rayleigh-Faber-Krahn inequality is a special case ($\beta=0$) of the previous theorem. It is not immediately clear what the analogous result should be in the Neumann case. The Laplacian without magnetic field ($\beta=0$) offers no guidance since its first eigenvalue is $0$ for every domain $\Omega$. The question was studied by Fournais and Helffer, who combined results from asymptotic analysis (as exposed in \cite{FournaisHelffer}) with isoperimetric-type inequalities to find several upper and lower bounds \cite{FH2019Magnetic}. In particular, they obtained the following result.

\begin{thm}[Fournais and Helffer, 2019]\label{thmAsympt} If $\Omega$ is an open, bounded and simply connected set with a $C^\infty$ smooth boundary, and if $\Omega$ is not a disk, then there exist $\beta_\ell(\Omega)>0$ and $\beta_u(\Omega)>0$ such that the inequality 
\begin{equation*}
	\lambda_1^N(\Omega,\beta)<\lambda_1^N(\Omega^*,\beta)
\end{equation*}  
holds for all $0<\beta<\beta_\ell(\Omega)$ and all $\beta>\beta_u(\Omega)$.
\end{thm} 
This suggests the  following reverse Faber-Krahn inequality. 
\begin{conj}[Fournais and Helffer, 2019] \label{conjRFK} Let $\Omega\subset\mathbb R^2$ be an open, bounded and simply connected set with a $C^\infty$ smooth boundary. Then, for all $\beta>0$, 
\begin{equation*}
	\lambda_1^N(\Omega,\beta)\le\lambda_1^N(\Omega^*,\beta),
\end{equation*}
with equality if, and only if, $\Omega$ is a disk.
\end{conj}
Equally importantly for our purpose, Fournais and Helffer emphasized the usefulness of choosing $A_\Omega$ as a vector potential for studying this problem.

While we were completing the present paper, we learned of the recent work of Kachmar and Lotoreichik \cite{KL2024Magnetic}, who also tackled Conjecture \ref{conjRFK} and proved an inequality of the form (using our notation) 
\begin{equation}\label{eqKL1}
\lambda_1^N(\Omega,\beta)\leq C(\Omega)\lambda_1^N(B_{\rho(\Omega)},\beta),
\end{equation}
 which holds for $\beta|\Omega|\le\beta^*\pi$ and for all convex planar domains with a $C^\infty$ smooth boundary. The radius $\rho(\Omega)$ is defined by the property that the torsion functions of $\Omega$ and $B_{\rho(\Omega)}$ have the same maximum. It follows from the Talenti Comparison Principle that $B_{\rho(\Omega)}\subset\Omega^*$. The authors then prove that, assuming $\beta>0$, the function $R\mapsto \lambda_1^N(B_R,\beta)$ is strictly increasing for $R$ small enough and they deduce the inequality
 \begin{equation}\label{eqKL2}
 \lambda_1^N(\Omega,\beta)\leq C(\Omega)\lambda_1^N(\Omega^*,\beta),
 \end{equation}
with equality only in the case of disk. The constant $C(\Omega)$ depends quite explicitly on $\Omega$, and can be estimated in some cases. For example, the authors are able to show that $C(\Omega)=1$ if $\Omega$ is an ellipse. Inequality \eqref{eqKL2} then validates Conjecture \ref{conjRFK} for  $\beta|\Omega|\le\beta^*\pi$ and Inequality \eqref{eqKL1} gives a quantitative version of it. However, they also prove that $C(\Omega)\geq 1$ and find domains $\Omega$ for which $C(\Omega)>1$, showing that  Inequality \eqref{eqKL2} is in general weaker than the conjectured one.
 
The authors' proof of Inequality \eqref{eqKL1} has points in common with our methods. It consists in using a suitable test function which is constant on the level lines of the torsion function on $\Omega$, along with estimates on the area and the perimeter of the super-level sets of the torsion function. 

\subsection{Main result and outline of the paper}
\label{secOutline}

We give a partial confirmation of Conjecture \ref{conjRFK}, for the magnetic field strength $\beta$ below a certain threshold. Let us point out that, contrary to Theorem \ref{thmAsympt}, our threshold for $\beta$ only depends on the area of $|\Omega|$ and on the constant $\beta^*$ introduced in Definition \ref{dfnBetaStar}. 

\begin{thm} \label{thmRFK} Let $\Omega\subset\mathbb R^2$ be an open, bounded and simply connected set with a $C^\infty$ smooth boundary. Assume that $\beta>0$ is such that $\beta|\Omega|\le \beta^*\pi$. Then 
\begin{equation}\label{eqRFK}
	\lambda_1^N(\Omega,\beta)\le\lambda_1^N(\Omega^*,\beta),
\end{equation}
and equality occurs if, and only if, $\Omega$ is a disk.
\end{thm}

The proof is based on Rayleigh's principle and consists of two main parts. Since the conjectured inequality becomes an equality when $\Omega$ is a disk, we have to define a class of trial functions which, in the case of a disk, contains the first magnetic eigenfunction. This is the goal of Section \ref{secLevelLines}. As suggested by \cite{FH2019Magnetic}, we use the potential $A_\Omega$ derived from the torsion function $\psi$ of $\Omega$. We take as trial functions those which are constant on the level lines of $\psi$, that is, functions of the form $u(p)=g(t)$, where $p=(x,y)\in \Omega$ and $t=\psi(p)$. These can be understood as a substitute for radial functions in the case of a general simply connected domain. This approach is an example of the so-called \emph{method of prescribed level lines} (presented for instance in the classical monograph \cite{PolyaSzegoIsoperimetric}). Since our objective is to compare different domains having the same area, it is convenient to take as independent variable for these functions the area $a$ enclosed by the level lines of the torsion function. More explicitly, we perform the change of variable
\[a=\mu(t):=|\left\{p\in\Omega\,:\,\psi(p)>t\right\}|,\]
where $|\cdot|$ stands for the Lebesgue measure. Then, we have
\[u(p)=g(t)=f(a)\]
or, more explicitly,
\[u(p)=g(\psi(p))=f(\mu(\psi(p))),\]
with $f$ defined on the interval $[0,a^*]$, where $a^*:=|\Omega|$. Let us stress that the same idea is used in many existing works.  Among numerous examples, we can direct the reader to \cite[Section II]{PW1961} or to the definition of the decreasing rearrangement of a function in \cite[Section 2]{T76}.

Using this strategy, we are reduced to the study of functions on a fixed interval. We then express the Rayleigh quotient of $u$ in terms of $f$ and obtain
\begin{equation}
\frac{\int_\Omega\left|\nabla^{\beta A_\Omega} u(p)\right|^2\,dp}{\int_\Omega |u(p)|^2\,dp}=\frac{\int_0^{a^*}\left(a\,G_\Omega(a)|f'(a)|^2+\frac{\beta^2a}{G_\Omega(a)}|f(a)|^2\right)\,da}{\int_0^{a^*}|f(a)|^2\,da}.
\end{equation}
In the previous formula, 
\[G_\Omega:[0,a^*]\to [0,\infty]\] 
is a function (defined in Equations \eqref{eqGamma} and \eqref{eqG}) which depends only on $\Omega$. We show in Propositions \ref{propIsoperimetric} and \ref{propIsoperimetricG} that, as a result of the isoperimetric inequality, $G_\Omega(a)\ge 4\pi$ almost everywhere, with equality if, and only if, $\Omega$ is a disk. Now, setting
\begin{equation}\label{eqVarIntro} 
 \kappa_1(G_\Omega,\beta):=\inf_{f\in\mathcal F_\Omega\setminus\{0\}}\frac{\int_0^{a^*}\left(a\,G_\Omega(a)|f'(a)|^2+\frac{\beta^2a}{G_\Omega(a)}|f(a)|^2\right)\,da}{\int_0^{a^*}|f(a)|^2\,da},
\end{equation}
 we obtain
\[\lambda_1^N(\Omega,\beta)\le\kappa_1(G_\Omega,\beta).\]
The class of admissible function $\mathcal F_\Omega$ is specified in Section \ref{secArea}.

The idea behind the second part of our proof is to study the one-dimensional variational problem \eqref{eqVarIntro} directly, disregarding its origin from the method of level lines. Hence, in the course of Section \ref{secSL}, we only assume $G:[0,a^*]\to [0,\infty]$ to be a measurable function satisfying 
\[4\pi\le G(a)<\infty\]
for almost every $a$. Our goal is then to prove that $\kappa_1(G,\beta)\le\kappa_1(4\pi,\beta)$, with equality if, and only if, $G(a)=4\pi$ almost everywhere. We show that we can reduce the question to the case where $G$ is in $L^\infty$. Then, $\kappa_1(G,\beta)$ turns out to be the first eigenvalue of the Sturm-Liouville problem:
\begin{equation}
\label{eqSLIntro}
	\left\{\begin{array}{rl}
		-\left(P(a)f'\right)'+\beta^2Q(a)f=&\kappa f\mbox{ in }]0,a^*[,\\
		\lim_{a\to 0}P(a)f'(a)=&0,\\
		P(a^*)f'(a^*)=&0,
	\end{array}
\right.	
\end{equation}
with
\begin{equation*}
P(a):=a\,G(a) \mbox{\ \ \ and\ \ \ } Q(a):=\frac{a}{G(a)}.
\end{equation*}
Let us note that Problem \eqref{eqSLArea} is a special case ($G(a)=4\pi$) of Problem \eqref{eqSLIntro}. We finally show that $\kappa_1(G,\beta)$ decreases when $G$ increases pointwise (Proposition \ref{propMonotonicity}) by reducing the proof to the computation of a derivative. More precisely, given two $L^\infty$-functions $G_0$ and $G_1$, not almost-everywhere equal, such that
\[4\pi\le G_0(a)\le G_1(a)\] 
almost everywhere, we define 
\[G_z(a):=(1-z)\,G_0(a)+z\,G_1(a)\]
and $\kappa(z):=\kappa_1(G_z,\beta)$. We prove the Feynman-Hellmann type formula
\begin{equation*}
\kappa'(z)=\int_0^{a^*}\left(|a\,G_z(a)\,f_z'(a)|^2-\beta^2a^2|f_z(a)|^2\right)\frac{\delta G(a)}{aG_z(a)^2}\,da,
\end{equation*}
where $f_z$ is a normalized eigenfunction associated with $\kappa(z)$ and 
\[\delta G(a):=G_1(a)-G_0(a)\]
(Proposition \ref{propDerivative}). We then show that $\kappa'(z)<0$, by an analysis of the eigenvalue equation (Proposition \ref{propBounds}). As a result, we find
 \[\kappa_1(G_0,\beta)=\kappa(0)>\kappa(1)=\kappa_1(G,\beta),\]
proving monotonicity. 

We thus obtain 
\[\lambda_1^N(\Omega,\beta)\le\kappa_1(4\pi,\beta)\]
with equality if, and only if, $\Omega$ is a disk. We conclude the proof of Theorem \ref{thmRFK} by noting that, according to  the discussion at the end of Section \ref{subsecDisk}, $\kappa_1(4\pi,\beta)=\lambda_1^N(\Omega^*,\beta)$ when $\beta|\Omega|\le \beta^*\pi$. 

\begin{rem}
We are using only once each of the hypotheses that $\Omega$ is simply connected and that $\beta|\Omega|\le\beta^*\pi$. We are actually proving that, for any open, bounded and connected set  $\Omega\subset\mathbb R^2$ and for any $\beta>0$,
\begin{equation}\label{eqGeneralizedRFK}
	\lambda_1^N(\Omega,\beta\,A_\Omega)\le\kappa_1(4\pi,\beta),
\end{equation}
with equality if, and only if, $\Omega$ is a disk, where $A_\Omega$ is the vector potential defined in Equation \eqref{eqVectPot}. The left-hand side of Inequality \eqref{eqGeneralizedRFK} is equal to $\lambda_1^N(\Omega,\beta)$ if $\Omega$ is simply connected, and the right-hand side is equal to $\lambda_1(\Omega^*,\beta)$ if $\beta|\Omega|\le\beta^*\pi$.
\end{rem}

Some technical details, needed to deal with regularity issues, are relegated to Appendices \ref{appChainRules} and \ref{appSL}. 

\section{Method of level lines}
\label{secLevelLines}

\subsection{Preliminaries on the torsion function}
\label{subsecTorsion}

We recall that, throughout the paper, $\Omega\subset\mathbb R^2$ is assumed to be an open, bounded and simply connected set, with $C^\infty$ smooth boundary. Standard regularity theory for elliptic PDEs implies that the torsion function $\psi$, defined through the boundary-value problem \eqref{eqBVPTorsion}, belongs to $C^\infty(\overline{\Omega})$ and is real-analytic inside $\Omega$. Moreover, $\psi$ is by definition superharmonic and  therefore, by the strong minimum principle, $\psi(p)>0$ for all $p=(x,y)\in\Omega$. For future reference, we set 
\[t^*:=\max_{p\in\overline{\Omega}} \psi(p),\]
so that
\[\psi(\overline{\Omega})=[0,t^*].\]

We recall that the zero-set of a real-analytic function, defined in an open and connected domain and not identically zero, has zero Lebesgue measure (see for instance \cite{Mityagin2022Nodal} for a short self-contained proof). It follows that the set of critical points
\[\mathcal C:=\left\{p\in\Omega\,:\, \nabla\psi(p)=0\right\}\]
has zero Lebesgue measure (as can be seen by applying the previous result to $\left|\nabla\psi\right|^2$). 
 
\subsection{Torsion-type functions}

Let us recall that we want to bound $\lambda_1^N(\Omega,\beta)$ from above by the Rayleigh quotient of a suitable trial function, computed using the vector potential $A_\Omega$. We limit our choice to a class of trial functions that we call \emph{torsion-type}, which are, by definition, constant on the level-sets
\[\Gamma_t:=\left\{p\in\Omega\,:\,\psi(p)=t\right\}.\]
This procedure is known as the \emph{method of level sets} or, in the present two-dimensional case, \emph{method of level lines}. Intuitively, we can write such a torsion-type function $u:\overline{\Omega}\to\mathbb C$ as
\[u(p)=g(\psi(p)),\]
where  $g$ is a function from $[0,t^*]$ to $\mathbb C$. To obtain the best possible upper-bound, we then minimize over the function $g$. We have thus replaced the original question with a one-dimensional variational problem.   

In order to proceed, we need to clarify a technical detail: what regularity is required of $g$ for $u$ to be in $H^1(\Omega)$? This requires some basic tools from geometric measure theory and the method of rearrangements. The current section and the next follow rather closely  \cite{BLL2012Isoperimetric} and mainly consists of restatements of classical results. We also use these tools to compute the Rayleigh quotient of a torsion-type function.

We define, for $t\in [0,t^*]$, the  super-level set
\begin{equation*}
	\Omega_t:=\{p\in\Omega\,:\,\psi(p)>t\}
\end{equation*}
and we set
\begin{equation*}
	\mu(t):=|\Omega_t|.
\end{equation*}
For future reference, we write $a^*:=|\Omega|$. The function 
$\mu:[0,t^*]\to[0,a^*]$ is called the \emph{distribution function} of $\psi$. It is non-increasing and right-continuous, with $\mu(0)=a^*$ and $\mu(t^*)=0$ (see \cite{BLL2012Isoperimetric}).

The following result, which is a special case of a fundamental theorem in geometric measure theory, is our starting point (see for instance \cite[Section 3.4.3]{EvansGariepy}). It would also hold if $\psi$ were only Lipschitz.
\begin{prop}[The Coarea Formula]\label{propCoarea} Let $h:\Omega\to\mathbb C$ be integrable. 
\begin{enumerate}
\item For almost all $t\in[0,t^*]$ (with respect to the Lebesgue measure), the restriction of $h$ to $\Gamma_t$ is $\mathcal H^1$-integrable (where  $\mathcal H^1$ is the $1$-dimensional Hausdorff measure).
\item  The function
\[t\mapsto \int_{\Gamma_t}h(q)\,d\,\mathcal H^1(q)\]
is integrable in $[0,t^*]$.
\item For all $t\in[0,t^*]$,
\begin{equation*}
	\int_{\Omega_t}h(p)|\nabla \psi(p)|\,dp=\int_t^{t^*}\left(\int_{\Gamma_s}h(q)\,d\,\mathcal H^1(q)\right)\,ds.
\end{equation*}
\end{enumerate}
\end{prop}

\begin{cor} \label{corMu}Let us define (almost everywhere) the function
\[\gamma_\Omega:[0,t^*]\to[0,\infty]\]
by 
\begin{equation}
\label{eqGamma}
\gamma_\Omega(t):=\int_{\Gamma_t}\frac{1}{|\nabla\psi(q)|}\,d\,\mathcal H^1(q).
\end{equation}
Then $\gamma_\Omega\in L^1(]0,t^*[,dt)$ and 
\begin{equation}
\label{eqDerivative}
	\mu(t)=\int_t^{t^*}\gamma_\Omega(s)\,ds.
\end{equation}
\end{cor}

\begin{proof}[Proof of the corollary] Roughly speaking, we apply the Coarea Formula to the function $g=1/|\nabla \psi|$. However, this function may not be integrable. To bypass this difficulty, we define, for any positive integer $n$,
\[h_n(p):=\min\left(n,\frac1{|\nabla \psi(p)|}\right).\]
Each function $h_n$ is bounded and therefore integrable in $\Omega$. The Coarea Formula then yields, for all $t\in[0,t^*]$,
\begin{equation*}
	\int_{\Omega_t}h_n(p)|\nabla \psi(p)|\,dp=\int_t^{t^*}\left(\int_{\Gamma_s}h_n(q)\,d\sigma_s(q)\right)\,ds.
\end{equation*}

By applying twice the Monotone Convergence Theorem, we find that the right-hand side converges to 
\[\int_t^{t^*}\gamma_\Omega(s)\,ds\]
as $n$ tends to $\infty$. On the other hand, $h_n(p)|\nabla\psi(p)|$ is ultimately equal to $1$ if $1/|\nabla\psi(p)|$ is finite, that is to say if $p$ is not a critical point, and is constantly $0$ if $p$ is a critical point. By the Monotone Convergence Theorem, the left-hand side  therefore converges to 
\[|\Omega_t\setminus\mathcal C|\]
(we recall that $\mathcal C$ is the set of critical points of $\psi$). Since $\mathcal C$ has measure zero, the previous quantity is equal to $\mu(t)$.

It remains only to show that $\gamma_\Omega$ belongs to $L^1(]0,t^*[,dt)$. Since $\gamma_\Omega$ is non-negative, it is enough to apply Equation \eqref{eqDerivative} for $t=0$:
\[\int_0^{t^*}\gamma_\Omega(s)\,ds=\mu(0)=a^*<\infty.\qedhere\]
\end{proof}

\begin{rem}\label{remCoarea}
Adapting the proof of the previous corollary, we see that the Coarea Formula, that is, part 3 of Proposition \ref{propCoarea}, holds  when $h$ is any measurable function from $\Omega$ to $[0,\infty]$ which is finite almost everywhere, without any conditions of integrability. In that case, either both side of the equation are finite and equal or both are infinite.
\end{rem}

\begin{rem}\label{remAC}
According to Equation \eqref{eqDerivative}, $\mu$ is absolutely continuous, with derivative $-\gamma_\Omega$. 
\end{rem}

It is crucial to note that the isoperimetric inequality implies a lower bound on $\gamma_\Omega$.
\begin{prop}\label{propIsoperimetric}
 For almost every $t\in [0,t^*]$, $\gamma_\Omega(t)\ge 4\pi$. Furthermore, we have $\gamma_\Omega(t)=4\pi$ almost everywhere if, and only if, $\Omega$ is a disk. 
\end{prop}

\begin{proof} We first use the Cauchy-Schwarz inequality: for $t\in[0,t^*]$,
\begin{align*} 
|\Gamma_t|&=\int_{\Gamma_t}1\,d\,\mathcal H^1=\int_{\Gamma_t}|\nabla \psi|^{1/2}\frac{1}{|\nabla \psi|^{1/2}}\,d\,\mathcal H^1\\
&\le \left(\int_{\Gamma_t}|\nabla \psi|\,d\,\mathcal H^1\right)^{1/2}\left(\int_{\Gamma_t}\frac{1}{|\nabla \psi|}\,d\,\mathcal H^1\right)^{1/2}.
\end{align*}
According to Sard's theorem, almost all $t$ in $[0,t^*]$ are regular values of $\psi$. For these values, the Hausdorff measure $\mathcal H^1$ coincides on $\Gamma_t$ with the measure $\sigma_t$ induced by the arclength and we obtain, using Green's formula, 
\begin{equation*}
\int_{\Gamma_t}|\nabla\psi(q)|\,d\sigma_t(q)=-\int_{\Gamma_t}\frac{\partial\psi}{\partial \nu}(q)\,d\sigma_t(q)=\int_{\Omega_t}\Delta \psi(p)\,dp=|\Omega_t|.
\end{equation*} 
From the two previous facts and the isoperimetric inequality,
\begin{equation*}
	\gamma_\Omega(t)\ge \frac{|\Gamma_t|^2}{|\Omega_t|}\ge 4\pi,
\end{equation*}
which proves the desired inequality.

To treat the case of equality, we observe that if we have $\gamma_{\Omega}(t)=4\pi$  for almost every $t$, we also have, from the previous inequalities, $|\Gamma_t|^2/|\Omega_t|=4\pi$. Thus, $\Omega_t$ is a disk for almost every $t$ and, in particular, we can pick a positive sequence $(t_n)$, decreasing to $0$, such that $\Omega_{t_n}$ is a disk for all $n$. Since the sequence of sets $\Omega_{t_n}$ is non-decreasing with respect to inclusion and since $\Omega=\bigcup_n\Omega_{t_n}$, it follows that $\Omega$ is a disk. Conversely, if $\Omega$ is a disk, then $\psi$ is radial, all level sets are circles on which $|\nabla\psi|$ is constant, hence we have equality in both the Cauchy-Schwarz and the isoperimetric inequalities. 
\end{proof}

\begin{dfn}\label{dfnG} Let 
 $\mathcal G_\Omega$ be the space of functions $g:[0,t^*]\to\mathbb C$ (up to equality almost everywhere) such that 
\begin{enumerate}
	\item $g\in L^2(]0,t^*[,\gamma_\Omega(t)\,dt)$,
	\item $g'\in L^2(]0,t^*[, \mu(t)\,dt)$.
\end{enumerate}
\end{dfn}

Let us briefly discuss the meaning of Definition \ref{dfnG}. The first condition, together with the inequality in Proposition \ref{propIsoperimetric}, implies that $g$ is locally integrable in $]0,t^*[$. We can therefore consider the associated distribution in $]0,t^*[$, which we also denote by $g$. The second condition then requires $g'$, the derivative of this distribution, to be  a locally integrable function which satisfies
\begin{equation}\label{eqeqL2t}
	\int_0^{t^*}|g'(t)|^2\,\mu(t)\,dt<\infty.
\end{equation}

\begin{prop} \label{propH1} If $g$ belongs to $\mathcal G_\Omega$, the function $u:\Omega\to\mathbb C$, defined by \[u(p):=g(\psi(p)),\] belongs to $H^1(\Omega)$. 
\end{prop}

\begin{proof}  

We first apply the Coarea Formula (specifically Remark \ref{remCoarea}) to the function $|u|^2$ and obtain
\begin{equation}
	\label{eqL2t}
	\int_\Omega |u(p)|^2\,dp=\int_0^{t^*}|g(t)|^2\,\gamma_\Omega(t)\,dt.
\end{equation}
By definition of $\mathcal G_\Omega$, the right-hand side is finite, and therefore $u\in L^2(\Omega)$.

We now claim that the function $u$ is locally integrable in $\Omega$ and that, interpreted as a distribution in $\Omega$,
\begin{equation}\label{eqGrad}
	\nabla u(p)=g'(\psi(p))\nabla \psi(p).
\end{equation}
If $g\in C^1([0,t^*])$, the claim follows immediately from the chain rule. When we only assume that $g\in\mathcal G_\Omega$, we can establish \eqref{eqGrad} using an approximation argument detailed in Appendix \ref{appChainRules}. 

We finally apply the Coarea Formula to $|\nabla u|^2$, using the previous claim, and obtain
\begin{equation*}	
\int_\Omega\left|\nabla u(p)\right|^2\,dp=\int_\Omega |g'(\psi(p))|^2|\nabla \psi(p)|^2\,dp=\int_0^{t^*}|g'(t)|^2\left(\int_{\Gamma_t}|\nabla \psi(q)|\,d\,\mathcal H^1(q)\right)\,dt.
\end{equation*}
As shown in the proof of Proposition \ref{propIsoperimetric},
\begin{equation*}
	\int_{\Gamma_t}|\nabla \psi(q)|\,d\,\mathcal H^1(q)=\mu(t)
\end{equation*}
for almost every $t\in]0,t^*[$, therefore
\begin{equation*}
	\int_\Omega\left|\nabla u(p)\right|^2\,dp=\int_0^{t^*}|g'(t)|^2\,\mu(t)\,dt.
\end{equation*}
The right-hand side is finite by definition of $\mathcal G_\Omega$, therefore $\nabla u$ is in $L^2(\Omega)$ and thus $u$ is in $H^1(\Omega)$.
\end{proof}

Let us finally give a formula for the $L^2$-norm of the magnetic gradient of a torsion-type function, using the vector potential $A_\Omega$. We  fix $g\in\mathcal G_\Omega$ and we define, as before, $u(p):=g(\psi(p))$. For any $\beta\ge0$, we find, using Equation \eqref{eqGrad}, that
\begin{equation*}
	\nabla u(p)-i\beta u(p)A_\Omega(p)=g'(\psi(p))\nabla \psi(p)+i\beta g(\psi(p))\nabla^\perp\psi(p)
\end{equation*}
for almost every $p$ in $\Omega$. Since the vectors $\nabla\psi(p)$ and $\nabla^\perp \psi(p)$ are orthogonal and have the same length,
\begin{align*}
	\left|g'(\psi(p))\nabla \psi(p)-i\beta g(\psi(p))\nabla^\perp\psi(p)\right|^2=\left(|g'(t)|^2+\beta^2|g(t)|^2\right)|\nabla\psi(p)|^2,
\end{align*}
with $t=\psi(p)$. The Coarea Formula then gives
\begin{equation}\label{eqMagnG}
	\int_\Omega\left|\nabla u(p)-i\beta u(p)A_\Omega(p)\right|^2\,dp=\int_0^{t^*}\left(|g'(t)|^2+\beta^2|g(t)|^2\right)\,\mu(t)\,dt.
\end{equation}  

\subsection{Reparametrization by the area}
\label{secArea}

Since we are comparing different domains having the same area, it is convenient to use an alternative parametrization of torsion-type functions. We use as a new variable $a=\mu(t)$, that is, we parametrize the values $u(p)$ of a torsion-type function $u$ by the area of the super-level set $\Omega_t$, with $t=\psi(p)$. By definition of the distribution function $\mu$, the variable $a$ runs over the  interval $[0,a^*]$. Roughly speaking, we can then write a torsion-type function $u:\Omega\to\mathbb C$ as $u(p)=f(a)$, where $a=\mu(\psi(p))$ and $f$ is a complex-valued function defined on $[0,a^*]$. 

As before, we need to specify the regularity required of $f$ for $u$ to be in $H^1(\Omega)$. We do this using Proposition \ref{propH1} and the properties of the distribution function $\mu$. Let us first note that, since $\mu$ is absolutely continuous, with $\mu'(t)$ negative almost everywhere ($\mu'(t)=-\gamma_\Omega(t)\le-4\pi$), $\mu$ is strictly decreasing and therefore is a homeomorphism from $[0,t^*]$ to $[0,a^*]$. The formulas $f=g\circ\mu^{-1}$ and $g=f\circ\mu$ are therefore well defined and maps complex-valued functions $g$ on $[0,t^*]$ to complex-valued functions $f$ on $[0,a^*]$ and vice versa. In addition, $\mu^{-1}$ is also absolutely continuous, with derivative almost everywhere $(\mu^{-1})'(a)=-1/G_\Omega(a)$, where we define
\begin{equation}\label{eqG}
G_\Omega(a):=\gamma_\Omega\left(\mu^{-1}(a)\right).
\end{equation} 
Since both $\mu$ and $\mu^{-1}$ are absolutely continuous, they map sets of measure zero to sets of measure zero. This means that the above correspondence between the functions $g(t)$ and $f(a)$ can be defined up to  equality almost everywhere.

Proposition \ref{propIsoperimetric} can be immediately translated in terms of the function $G_\Omega$.
\begin{prop}\label{propIsoperimetricG}
For almost every $a\in [0,a^*]$,
\begin{equation}\label{eqLowerBoundG}
	G_\Omega(a)\ge 4\pi.
\end{equation}
In addition, equality holds for almost every $a$ if, and only if, $\Omega$ is a disk. 
\end{prop}

In the rest of this section, we frequently use the following change of variable formula. It is simply a consequence of the fact that, since $-\gamma_\Omega$ is the derivative of $\mu$, the measure $\gamma_\Omega(t)\,dt$ on $[0,t^*]$ is the pull-back by the mapping $\mu:[0,t^*]\to[0,a^*]$ of the Lebesgue measure on $[0,a^*]$.  
\begin{prop}\label{propChangeVar}
	If $h$ is a measurable function from $[0,a^*]$ to $[0,\infty]$, then
	\begin{equation*}
		\int_0^{a^*}h(a)\,da=\int_0^{t^*}h(\mu(t))\,\gamma_\Omega(t)\,dt.
	\end{equation*} 
	Furthermore, if $h$ is a measurable function from $[0,a]$ to $\mathbb C$, then $h$ is integrable in $[0,a^*]$ if, and only if, $t\mapsto h(\mu(t))\,\gamma_{\Omega}(t)$ is integrable in $[0,t^*]$, and in that case the above identity also holds.
\end{prop}

\begin{dfn}\label{dfnF} Let $\mathcal F_\Omega$ be the space of functions $f:[0,a^*]\to\mathbb C$ (up to equality almost everywhere) such that 
\begin{enumerate}
	\item $f\in L^2(]0,a^*[, da)$,
	\item $f'\in L^2(]0,a^*[, a\,G_\Omega(a)\,da)$.
\end{enumerate}
\end{dfn}

\begin{prop}\label{propGF}
 The function $f:[0,a^*]\to\mathbb C$ belongs to $\mathcal F_{\Omega}$ if, and only if, $g(t)=f(\mu(t))$ belongs to $\mathcal G_\Omega$.
\end{prop}
Before proceeding with the proof, let us point out that the characterization of the functions $f$,  such that the corresponding $u$ is in $H^1$,  is an immediate corollary of Proposition \ref{propGF} and Proposition \ref{propH1}. 
\begin{cor}\label{corH1}
 If $f$ belong to $\mathcal F_\Omega$, the function $u:\Omega\to\mathbb C$, defined by 
 \[u(p)=f(\mu(\psi(p))),\]
 belongs to $H^1(\Omega)$. 
 \end{cor}
 \begin{proof}[Proof of Proposition \ref{propGF}]  From Proposition \ref{propChangeVar}, with $h(a)=|f(a)|^2$, it follows that 
 \begin{equation}\label{eqL2F}
 \int_0^{a^*}|f(a)|^2\,da=\int_0^{t^*}|g(t)|^2\,\gamma_\Omega(t)\,dt
 \end{equation}
 and therefore that $f$ is in $L^2(]0,a^*[,da)$ if, and only if, $g$ is in $L^2(]0,t^*[,\gamma_\Omega(t)\,dt)$.
 
 In addition, we claim that $f'\in L^1_{\rm loc}([0,a^*[)$ if, and only if, $g'\in L^1_{\rm loc}(]0,t^*])$ and that, in this case, the following identity holds: for almost every $t\in]0,t^*[$,
 \begin{equation}\label{eqPrime}
		g'(t)=-\gamma_\Omega(t)\,f'(\mu(t))
 \end{equation}
(see Appendix \ref{appChainRules} for a proof). Using this, and Proposition \ref{propChangeVar} with $h(a)=|f'(a)|^2\,a\,G_\Omega(a)$, we find
\begin{equation*}
	\int_0^{a^*}|f'(a)|^2\,a\,G_\Omega(a)\,da=\int_0^{t^*}|f'(\mu(t))|^2\gamma_\Omega(t)^2\mu(t)\,dt=\int_0^{t^*}|g'(t)|^2\mu(t)\,dt.
\end{equation*}
We conclude that $f'$ is in $L^2(]0,a^*[,a\,G_\Omega(a)\,da)$ if, and only if, $g'$ is in $L^2(]0,t^*[,\mu(t)\,dt)$.
 \end{proof}
 
We still have to compute the $L^2$-norm of the magnetic gradient in the new variable $a$. Performing the change of variable $a=\mu(t)$ (using Proposition \ref{propChangeVar}) in the right-hand side of Equation \eqref{eqMagnG} and using Identity \eqref{eqPrime}, we find
\begin{align*}
	&\int_\Omega\left|\nabla u(p)-i\beta u(p) A_\Omega(p)\right|^2\,dp\\
	=&\int_0^{t^*}\left(|g'(t)|^2+\beta^2|g(t)|^2\right)\,\mu(t)\,dt\\
			=&\int_0^{t^*}\left(\gamma_\Omega(t)^2|f'(\mu(t))|^2+\beta^2|f(\mu(t))|^2\right)\frac{\mu(t)}{\gamma_\Omega(t)}\gamma_\Omega(t)\,dt\\
	=&\int_0^{a^*}\left(G_\Omega(a)^2|f'(a)|^2+\beta^2 |f(a)|^2\right)\frac{a}{G_\Omega(a)}\,da.
 \end{align*}  
We finally obtain
 \begin{equation}\label{eqMagnF}
 	\int_\Omega\left|\nabla u(p)-i\beta u(p) A_\Omega(p)\right|^2\,dp=\int_0^{a^*}\left(a\,G_\Omega(a)|f'(a)|^2+\frac{\beta^2a}{G_\Omega(a)}|f(a)|^2\right)\,da.
\end{equation}

\subsection{Reduction to a one-dimensional variational problem}

Let us now define 
\begin{equation}
\label{eqKappa}
	\kappa_1(G_\Omega,\beta):=\inf_{f\in\mathcal F_\Omega\setminus\{0\}}\frac{\int_0^{a^*}\left(a\,G_\Omega(a)|f'(a)|^2+\frac{\beta^2a}{G_\Omega(a)}|f(a)|^2\right)\,da}{\int_0^{a^*}|f(a)|^2\,da}.
\end{equation}

\begin{prop}\label{propUpper} For any $\beta\ge0$,
\begin{equation*}
	\lambda_1^N(\Omega,\beta)\le \kappa_1(G_\Omega,\beta).
\end{equation*}
\end{prop}

\begin{proof} For $f\in\mathcal F_\Omega\setminus\{0\}$, the torsion type function $u(p)=f(\mu(\psi(p)))$ is in $H^1(\Omega)$, according to Corollary \ref{corH1}, and its magnetic Rayleigh quotient satisfies
\begin{equation*}
\frac{\int_\Omega\left|\nabla u(p)-i\beta u(p)A_\Omega(p)\right|^2\,dp}{\int_\Omega |u(p)|^2\,dp}=\frac{\int_0^{a^*}\left(a\,G_\Omega(a)|f'(a)|^2+\frac{\beta^2a}{G_\Omega(a)}|f(a)|^2\right)\,da}{\int_0^{a^*}|f(a)|^2\,da},
\end{equation*}
according to Equations \eqref{eqL2t}, \eqref{eqL2F} and \eqref{eqMagnF}.

Since the torsion-type functions associated with a function in $\mathcal F_\Omega$ form a subspace of $H^1(\Omega)$, it follows from Rayleigh's principle that $\lambda_1^N(\Omega,\beta)\le \kappa_1(G_\Omega,\beta)$.
\end{proof}

\section{Analysis of a Sturm-Liouville problem}
\label{secSL}

\subsection{Generalization of the one-dimensional problem}

To study Problem \eqref{eqKappa}, we disregard its origin from the method of level lines. We consider it as a purely one-dimensional problem and we allow for functions $a\mapsto G(a)$ satisfying only the following assumption.
\begin{assum}\label{assumG}  The function $G:[0,a^*]\to [0,\infty]$ is measurable, and 
\[4\pi\le G(a)<\infty\]
for almost every $a\in [0,a^*]$.
\end{assum} 
We extend Definition \ref{dfnF} to such an arbitrary $G$ and indicate the corresponding subspace of $L^2(]0,a^*[,\mathbb C)$ with $\mathcal F_G$. We extend similarly Problem \eqref{eqKappa} and indicate the infimum with $\kappa_1(G,\beta)$.

We will need the following upper bound in the special case where $G(a)=4\pi$ almost everywhere.

\begin{lem}\label{lemDisk}
	For any $\beta>0$, 
	\begin{equation*}
		\kappa_1(4\pi,\beta)<\beta.
	\end{equation*}
\end{lem}

\begin{proof} We use the trial function 
\begin{equation*}
	f_\beta(a):=\exp\left(-\frac{\beta a}{4\pi}\right).
\end{equation*}
A direct computation shows that 
\begin{equation*}
\frac{\int_0^{a^*}\left(4\pi\,a\,|f_\beta'(a)|^2+\frac{\beta^2\,a}{4\pi}|f_\beta(a)|^2\right)\,da}{\int_0^{a^*}|f_\beta(a)|^2\,da}=\beta+\frac{4\pi\,a^*f_\beta'(a^*)f_\beta(a^*)}{\int_0^{a^*}|f_\beta(a)|^2\,da}<\beta.\qedhere
\end{equation*}
\end{proof}

\subsection{Associated Sturm-Liouville problem}

We start the proof by making an additional assumption. 

\begin{assum}\label{assumGBounded}
 The function $G$ is bounded.
\end{assum}

We show in Section \ref{subsecUnbounded} how the assumption can be removed to obtain an upper bound on $\kappa_1(G,\beta)$ without restriction on $G$.

It is easy to check that under Assumption \ref{assumGBounded}, the associated function space $\mathcal F_G$ is actually independent of $G$. Indeed, we prove the following more precise result in Appendix \ref{appSL}.

\begin{lem}\label{lemFBounded}  Let $G$ be a function satisfying Assumptions \ref{assumG} and \ref{assumGBounded}. Then, the sets $\mathcal F_{G}$ and $\mathcal F_{4\pi}$ are equal, meaning that $\mathcal F_G$ consists of the functions $f$ such that 
\begin{enumerate}
	\item $f\in L^2(]0,a^*[, da)$;
	\item $f'\in L^2(]0,a^*[, 4\pi a\,da)$.
\end{enumerate}
In addition, $\mathcal F_{4\pi}$ equipped with the natural norm
\begin{equation*}
	\|f\|_{\mathcal F_{4\pi}}^2:=\int_0^{a^*}|f(a)|^2\,da+4\pi\int_0^{a^*}|f'(a)|^2a\,da
\end{equation*}
 is compactly embedded in $L^2(]0,a^*[, da)$.
\end{lem}

Standard arguments from functional analysis and spectral theory (see Appendix \ref{appSL}) give the following result.

\begin{prop}\label{propMinSL} Let $a\mapsto G(a)$ be a function  satisfying Assumptions \ref{assumG} and \ref{assumGBounded}. Then, the infimum $\kappa_1(G,\beta)$ of Problem \eqref{eqKappa} is a minimum, and coincides with the first eigenvalue of the following Sturm-Liouville problem:
\begin{equation}
\label{eqSL}
	\left\{\begin{array}{rl}
		-\left(P(a)f'\right)'+\beta^2Q(a)f=&\kappa f\mbox{ in }]0,a^*[,\\
		\lim_{a\to0}P(a)f'(a)=&0,\\
		P(a^*)f'(a^*)=&0,
	\end{array}
\right.	
\end{equation}
where 
\begin{equation*}
P(a):=a\,G(a) \mbox{\ and\ } Q(a):=\frac{a}{G(a)}.
\end{equation*}
\end{prop}

\subsection{Properties of eigenvalues and eigenfunctions}

We recall some properties of the Sturm-Liouville problem \eqref{eqSL}. Those are standard in the case where the coefficient functions $Q$ and $P$ are regular, that is, $Q$ continuous, $P$ continuously differentiable and positive in $[0,a^*]$. They can be shown to hold in our case also. For the reader's convenience, we provide proofs in Appendix \ref{appSL}.
\begin{prop}\label{propSL}		The Sturm-Liouville problem \eqref{eqSL}, dependent on a function $G$, satisfying Assumptions \ref{assumG} and \ref{assumGBounded}, and on a parameter $\beta>0$, has the following properties.
\begin{enumerate}
	\item There exists an orthonormal basis $(f_k)_{k\ge1}$ of  $L^2(]0,a^*[,da)$, consisting of eigenfunctions for Problem \eqref{eqSL}, each associated with an eigenvalue $\kappa_k(G,\beta)$, such that $(\kappa_k(G,\beta))_{k\ge1}$ is a non-decreasing sequence going to $+\infty$.
	\item The eigenvalues $\kappa_k(G,\beta)$ are all simple.
	\item  The eigenfunctions $f_k$ are locally absolutely continuous in $]0,a^*]$. In addition, their so-called \emph{quasi-derivatives}, defined as \[f_k^{[1]}:a\mapsto P(a)f_k'(a)\]
	are  absolutely continuous in $[0,a^*]$.
	\item The eigenfunction $f_1$ can be chosen positive in $[0,a^*]$. In addition, $f_1$ is absolutely continuous in $[0,a^*]$. 	
\end{enumerate}
\end{prop}

\begin{rem} \label{remReal} Let us note that, since the coefficient functions $P$ and $Q$ are real-valued, the Sturm-Liouville operator in Problem \eqref{eqSL} maps real-valued functions to real-valued functions. It follows that we can chose the basis $(f_k)$ so that all its functions are real-valued. In the rest of the paper, we assume that this is the case.
\end{rem}
In the rest of this section, we focus on the first eigenvalue $\kappa_1(G,\beta)$. Our goal is to show monotonicity of $\kappa_1(\beta,G)$ with respect to $G$.

\begin{prop}\label{propMonotonicity}
	Let $G_0$ and $G_1$ be two functions satisfying Assumptions \ref{assumG} and \ref{assumGBounded} and let $\beta>0$. Let us also assume that
\begin{enumerate} 
\item $G_0(a)\le G_1(a)$ almost everywhere,
\item $G_0(a)<G_1(a)$ in a set of positive measure.
\end{enumerate}
Then $\kappa_1(G_0,\beta)>\kappa_1(G_1,\beta)$.
\end{prop}

We reduce the proof of Proposition \ref{propMonotonicity} to the computation of a derivative. To set this up, we define, for $z\in[0,1]$, the convex combination
\begin{equation}
\label{eqGz}
 G_z(a):=(1-z)G_0(a)+zG_1(a).
\end{equation}
Let us note that $G_z$ also satisfies Assumptions \ref{assumG} and \ref{assumGBounded}. To simplify notation, we write
\begin{equation*}
	\kappa(z):=\kappa_1(G_z,\beta)
\end{equation*}
and we denote by $f_z$ an associated eigenfunction, which we choose positive and $L^2$-normalized. We denote by $P_z$ and $Q_z$ the coefficient functions associated with $G_z$ and by $q_z$ the quadratic form
\begin{equation}
\label{eqQZ}
	q_z(f):=\int_0^{a^*}\left(P_z(a)|f'(a)|^2+\beta^2Q_z(a)|f(a)|^2\right)\,da.
\end{equation}

\begin{prop}\label{propDerivative}
The function $z\mapsto\kappa(z)$ is differentiable in $[0,1]$, and for all $z\in]0,1]$,
\begin{equation}\label{eqFH}
	\kappa'(z)=\frac{d}{dw}_{|w=z}q_w(f_z).
\end{equation}
\end{prop}

Equation \eqref{eqFH} is sometimes called the Feynman-Hellmann formula and holds in a variety of settings. Here, it can be justified by the fact that, in the terminology of T.~Kato, $q_z$ is an analytic family of forms of type $(a)$ with respect to the parameter $z$, so that each of the mappings $z\mapsto \kappa_k(G_z,\beta)$ is analytic in a domain of the complex plane containing the real segment $[0,1]$. The general theory can be found in \cite[VII-\S\,4]{Kato19995Book}, where the formula appears as (4.56) on page 408.

To find a more explicit formula, we differentiate with respect to $w$ under the integral sign in the expression of $q_w(f_z)$. We have
\begin{align*}	
	\partial_w P_w(a)=&a\delta G(a),\\
	\partial_w Q_w(a)=&-a\frac{\delta G(a)}{G_w(a)^2},
\end{align*}
where 
\begin{equation*}
	\delta G(a):=G_1(a)-G_0(a).
\end{equation*}
Using Proposition \ref{propDerivative}, we obtain
\begin{align*}
	\kappa'(z)
					=&\int_0^{a^*}\left(a\delta G(a)|f_z'(a)|^2-\beta^2a\frac{\delta G(a)}{G_z(a)^2}|f_z(a)|^2\right)\,da\\
					=&\int_0^{a^*}\left(|P_z(a)f_z'(a)|^2-\beta^2a^2|f_z(a)|^2\right)\frac{\delta G(a)}{aG_z(a)^2}\,da.\\
\end{align*}
Using the notation
\begin{align*}
	X_z(a):=&f_z(a),\\
	Y_z(a):=&f_z^{[1]}(a)=P_z(a)f_z'(a),
\end{align*}
we can summarize the previous results as
\begin{equation}\label{eqDerivativeXY}
\kappa'(z)=\int_0^{a^*}\left(|Y_z(a)|^2-\beta^2a^2|X_z(a)|^2\right)\frac{\delta G(a)}{aG_z(a)^2}\,da.
\end{equation}

Monotonicity is then a consequence of the following property.
\begin{prop}\label{propBounds} For any $z\in[0,1]$ such that $\kappa(z)<\beta$, 
\[|Y_z(a)|^2-\beta^2a^2|X_z(a)|^2<0\]
 for all $a\in]0,a^*[$.
\end{prop}
Before proving this proposition in the next section, let us check that it implies Proposition \ref{propMonotonicity}. We need the following lemma.

\begin{lem}\label{lemAPrioriBound}
If Proposition \ref{propBounds} holds, then  we have $\kappa(G,\beta)<\beta$ for any $G$ satisfying assumptions \ref{assumG} and \ref{assumGBounded}.  
\end{lem}

\begin{proof}
If $G(a)=4\pi$ for almost every $a$ in $]0,a^*[$, then $\kappa(G,\beta)=\kappa(4\pi,\beta)$ and since, according to Lemma \ref{lemDisk}, $\kappa(4\pi,G)<\beta$, we have the desired result. We can therefore assume from now on that $G(a)>4\pi$ holds in a set of positive measure.

Let us now define, for $z\in[0,1]$,
\begin{equation*}
	G_z(a):=4\pi(1-z)+z\,G(a)
\end{equation*}
(this is a special case of the definition \eqref{eqGz}, with $G_0=4\pi $ and $G_1=G$). We recall that the function 
$z\mapsto \kappa(z):=\kappa(G_z,\beta)$ is differentiable (hence continuous)  and that $\kappa(0)=\kappa(4\pi,\beta)<\beta$. 

Let us now assume, by contradiction, that $\kappa(G,\beta)=\kappa(1)\ge \beta$. Then, by continuity, there must exist $z_0\in]0,1]$ such that $\kappa(z_0)=\beta$ and $\kappa(z)<\beta$ for all $z\in[0,z_0[$. 

According to our running assumption on $G$, 
\[\delta G(a):=G(a)-4\pi\ge0\]
almost everywhere and  $\delta G(a)>0$ in a set of positive measure. From Formula \eqref{eqDerivativeXY} and Proposition \ref{propBounds}, we obtain $\kappa'(z)<0$ for $z\in[0,z_0[$. Hence, $\kappa(z_0)<\kappa(0)<\beta$, leading to a contradiction. \end{proof}

It is now easy to complete the proof of Proposition \ref{propMonotonicity}. We go back to the original definition \eqref{eqGz} of $G_z$, and repeat the end of the proof of Lemma \ref{lemAPrioriBound}. We find that $z\mapsto \kappa(z)$ is strictly decreasing in $[0,1]$ and therefore
\begin{equation*}	\kappa(G_0,\beta)=\kappa(0)>\kappa(1)=\kappa(G_1,\beta),
\end{equation*}
which is the desired conclusion.

\begin{rem} A posteriori, Lemma \ref{lemAPrioriBound} shows that the assumption $\kappa(z)<\beta$ in Proposition \ref{propBounds} was superfluous. However, this assumption plays a crucial role in the proof of Proposition \ref{propBounds} presented in the next section, so that we were not able to avoid the indirect reasoning presented here.
\end{rem}

\subsection{Proof of Proposition \ref{propBounds}} 

We first note that Proposition \ref{propBounds} is equivalent to 
\begin{equation*}
	\left|R_z(a)\right|<\beta a,
\end{equation*}
for all $z\in[0,1]$ such that $\kappa(z)<\beta$ and all $a\in ]0,a^*[$, where
\begin{equation*}
	R_z(a):=\frac{Y_z(a)}{X_z(a)}.
\end{equation*}
 The dependence in $z$ is irrelevant to the following analysis, so we fix $z\in[0,1]$ such that $\kappa(z)<\beta$ and simplify the notation accordingly, writing $G$, $P$, $Q$, $R$ and $\kappa$ for $G_z$, $P_z$, $Q_z$, $R_z$ and $\kappa(z)$. 

We now reformulate the Sturm-Liouville equation in the phase-plane. In the following proposition, the functions under consideration are locally absolutely continuous in $]0,a^*]$, so their derivatives are locally integrable functions. The differential equations hold almost everywhere.
\begin{prop}\label{propSystem}
 The vector-valued function $a\mapsto(X(a),Y(a))$ is a solution in $]0,a^*]$ of the first-order system
\begin{equation}\label{eqSystemODE}
\left\{
\begin{array}{cl}
		X'=&\frac1{P(a)}Y,\\
		Y'=&(\beta^2Q(a)-\kappa)X.
\end{array}
\right.
\end{equation}	
In addition, $a\mapsto (X(a),Y(a))$ satisfies $X(a^*)>0$,  $Y(a^*)=0$ and has a continuous extension at $a=0$, with $X(0)>0$ and $Y(0)=0$. Finally, we have $Y(a)=O(a)$ as $a\to 0$.
\end{prop}
The differential system \eqref{eqSystemODE} in the proposition is satisfied in $]0,a^*[$ as an immediate consequence of the definition of $X$ and $Y$ and of the Sturm-Liouville differential equation in \eqref{eqSL}. The behavior of $X$ and $Y$ at the endpoints $a=0$ and $a=a^*$ is more delicate (especially the left endpoint $a=0$). This part of the result is established in Appendix \ref{appSL}.

Taking Proposition \ref{propSystem} for granted, a straightforward computation implies that $R$ satisfies a Riccati equation in $]0,a^*]$.
\begin{prop}\label{propRiccati}
The function $R$ is a solution of a first-order, non-linear,  Riccati-type differential equation in $]0,a^*]$,
\begin{equation}\label{eqRiccatiODE}
R'=\beta^2 Q(a)-\kappa-\frac1{P(a)}R^2.
\end{equation}
In addition, $a\mapsto R(a)$ has a continuous extension at $a=0$, with $R(0)=0$, and $R(a^*)=0$.  
\end{prop}

We now proceed with the proof of Proposition \ref{propBounds}. The Riccati equation \eqref{eqRiccatiODE} can be written
\begin{equation}
\label{eqODEF}
	R'=F(a,R),
\end{equation}
with 
\begin{equation*}
	F(a,Z):=\beta^2Q(a)-\kappa-\frac{1}{P(a)}Z^2.
\end{equation*}
The proof then goes roughly as follows. If we define
\begin{equation*}
		Z_\pm(a):=\pm\beta a,
\end{equation*}
we find 
\begin{equation*}
	F(a,Z_\pm(a))=\beta^2Q(a)-\kappa-\frac{1}{P(a)}\beta^2a^2=\frac{\beta^2a}{G(a)}-\kappa-\frac{\beta^2a^2}{a G(a)}=-\kappa,
\end{equation*}
while 
\begin{equation*}
	Z_\pm'(a)=\pm \beta.
\end{equation*}
Since $\kappa$ is positive and, by hypothesis, $\kappa<\beta$, the functions $Z_+$ and $Z_-$ are respectively a strict supersolution and a strict subsolution of the ODE \eqref{eqODEF}. If the coefficient functions $1/P$ and $Q$ were continuous in $[0,a^*]$, this would be sufficient to reach the conclusion of Proposition \ref{propBounds}. Here we can only claim that these functions are locally integrable in $]0,a^*]$. In particular, $1/P(a)$ could behave as $1/a$ as $a\to 0$.  However, this method of comparison still works for Equation \eqref{eqRiccatiODE}. Let us state and prove the result more rigorously.
\begin{lem}\label{lemComparison}
	For all $a\in]0,a^*]$,
	\begin{equation}\label{eqComparison}
	-\beta a =Z_-(a)<R(a)<Z_+(a)=\beta a.
	\end{equation}
\end{lem}
\begin{proof} We first show that Inequality \eqref{eqComparison} holds in a neighborhood of $0$. Since $X(0)>0$ and $Y(a)=O(a)$, we have $R(a)=O(a)$ and thus
\begin{equation*}
	F(a,R(a))=-\kappa+O(a)
\end{equation*}
as $a\to0$. Since $0<\kappa<\beta$, it follows that there exists $\varepsilon>0$ such that, for all $a\in]0,\varepsilon[$,
\begin{equation*}
	|F(a,R(a))|<\beta.
\end{equation*}
From the differential equation \eqref{eqODEF}, we get
\begin{equation*}
	R(a)=\int_0^a F(\alpha,R(\alpha))\,d\alpha
\end{equation*}
and therefore, for $a\in]0,\varepsilon[$,
\begin{equation*}
	|R(a)|\le \int_0^a |F(\alpha,R(\alpha))|\,d\alpha <\beta a.
\end{equation*}

Let us now assume that there exists some $a\in]0,a^*]$ such that either 
\begin{equation*}
		R(a)>\beta a
\end{equation*}
or
\begin{equation*}
		R(a)<-\beta a.
\end{equation*}
We define $a_1$ as the infimum of  all such $a$'s.  From the first part, $a_1>0$, and since $R$ is continuous we have $R(a_1)=\beta a_1$ or $R(a_1)=-\beta a_1$.

Let us consider the first case and let us pick $a\in]0,a_1[$. We have 
\begin{equation}
	\label{eqCompUp}
	R(a)-\beta a=\int_{a}^{a_1}(\beta-F(\alpha,R(\alpha)))\,d\alpha.
\end{equation}
Since 
\begin{equation*}
		F(\alpha,\beta \alpha)=-\kappa,
\end{equation*}
for almost all $\alpha\in]0,a^*]$, we have
\begin{align*}
		&|F(\alpha,R(\alpha))+\kappa|=|F(\alpha,R(\alpha))-F(\alpha,\beta \alpha)|\\=&\frac{1}{P(\alpha)}\left|R(\alpha)^2-\beta^2\alpha^2\right| \le\frac{1}{4\pi \alpha}\left|R(\alpha)^2-\beta^2\alpha^2\right|.
\end{align*}
By continuity of $\alpha\mapsto R(\alpha)$, we have, for $\alpha$ close enough to $a_1$, that $F(\alpha,R(\alpha))$ is close to $-\kappa$, in particular $F(\alpha,R(\alpha))<0$, and thus
\begin{equation*}
	\beta-F(\alpha,R(\alpha))>0.
\end{equation*}
Using Equation \eqref{eqCompUp}, we find
\begin{equation*}
	R(a)>\beta a
\end{equation*}
for $a<a_1$ and $a$ close enough to $a_1$, contradicting the definition of $a_1$.

In the second case, we have, for $a\in]0,a_1[$, 
\begin{equation}
\label{eqCompLow}
	R(a)+\beta a=\int_{a}^{a_1}(-\beta-F(\alpha,R(\alpha)))\,d\alpha.
\end{equation}
As above, we obtain that for $\alpha$ close enough to $a_1$, $F(\alpha,R(\alpha))$ is close to $-\kappa$. In particular, since $\kappa<\beta$ by hypothesis,
\begin{equation*}
	-\beta-F(\alpha,R(\alpha))<0
\end{equation*}
for $\alpha$ close enough to $a_1$. It then follows from Equation \eqref{eqCompLow} that
\begin{equation*}
	R(a)<-\beta a
\end{equation*}
for $a<a_1$ and $a$ close enough to $a_1$, contradicting the definition of $a_1$.

Thus, Inequality \eqref{eqComparison} holds in all of $]0,a^*]$.
\end{proof}

\subsection{Upper bound for the general variational problem}
\label{subsecUnbounded}

We now complete our study of the one-dimensional variational problem \eqref{eqKappa} by removing Assumption \ref{assumGBounded}. Our goal is the following result.

\begin{prop} \label{propUpperBound}

Let $G$ be a function satisfying Assumption \ref{assumG}. For all $\beta>0$,
\[\kappa_1(G,\beta)\le \kappa_1(4\pi,\beta),\]
with equality if, and only if, $G(a)=4\pi$ for almost every $a$.
\end{prop}

For any positive integer $n$, we set  
\begin{equation*}
	G_n(a):=\min\{G(a),4n\pi\}.
\end{equation*}
Let us note that $G_n$ then satisfies both Assumptions \ref{assumG} and \ref{assumGBounded}, and that, for a given $a$,  $(G_n(a))_{n\ge1}$ is non-decreasing and tends to $G(a)$.  

\begin{lem}\label{lemLimit}
	The sequence $(\kappa_1(G_n,\beta))_{n\ge1}$ is non-increasing and 
	\begin{equation*}
		\lim_{n\to\infty}\kappa_1(G_n,\beta)=\kappa_1(G,\beta).
	\end{equation*}
\end{lem}

\begin{proof}  We recall that, according to Lemma \ref{lemFBounded}, $\mathcal F_{G_n}=\mathcal F_{4\pi}$ and $\mathcal F_{4\pi}$ is compactly embedded in $L^2(]0,a^*[,da)$. We also note that Assumption \ref{assumG} implies that $\mathcal F_G\subset\mathcal F_{4\pi}$. It follows from Proposition \ref{propMonotonicity} that the sequence $(\kappa_1(G_n,\beta))_{n\ge1}$ is non-increasing with $n$ and therefore convergent. Let us define
\begin{equation*}
	\kappa^*:=\lim_{n\to\infty} \kappa_1(G_n,\beta).
\end{equation*}
It remains to show that $\kappa^*=\kappa_1(G,\beta)$.

To simplify the notation slightly, we write
\begin{equation*}
	\kappa^n:=\kappa_1(G_n,\beta)
\end{equation*}
and, for any $f\in\mathcal F_{4\pi}$, 
\begin{equation*}
	q_n(f):=\int_0^{a^*}\left(P_n(a)|f'(a)|^2+\beta^2Q_n(a)|f(a)|^2\right)\,da,
\end{equation*}
where
\begin{equation*}
	P_n(a):=a\,G_n(a)
\end{equation*}
and
\begin{equation*}
	Q_n(a):=\frac{a}{G_n(a)}.
\end{equation*}
Let us now fix $f\in\mathcal F_G$ with $\|f\|_{L^2}=1$. By definition of $\kappa_1(G_n,\beta)$, 
\begin{equation}
\label{eqKappaN}
	q_n(f)\ge \kappa_1(G_n,\beta).
\end{equation} 
We find that 
\begin{equation*}
	\int_0^{a^*}P_n(a)|f'(a)|^2\,da\to\int_0^{a^*}P(a)|f'(a)|^2\,da
\end{equation*}
and
\begin{equation*}
	\int_0^{a^*}Q_n(a)|f(a)|^2\,da\to\int_0^{a^*}Q(a)|f(a)|^2\,da
\end{equation*}
from the Monotone and the Dominated Convergence Theorems, respectively. Passing to the limit in Inequality \eqref{eqKappaN}, we obtain 
\begin{equation*}
	q(f)\ge \kappa^*.
\end{equation*}
Therefore $\kappa^*\le \kappa_1(G,\beta)$. 

Let now $f_1^n$ denote, for each $n$, a function in $\mathcal F_{4\pi}$ such that $\|f_1^n\|_{L^2}=1$ and $q_n(f_1^n)=\kappa^n$. Since the sequence $(\kappa^n)$ is convergent, we can easily check, using Assumption \ref{assumG}, that $(f_1^n)$ is bounded in $\mathcal F_{4\pi}$. Hence, we can find a subsequence $(f_{1}^{n_k})$ and a function $f^*$ in $\mathcal F_{4\pi}$ such that $(f_{1}^{n_k})$ converges to $f^*$ weakly in $\mathcal F_{4\pi}$ and strongly in $L^2(]0,a^*[,da)$. Now, for a fixed integer $m$, as soon as $n_k\ge m$,
\begin{align*}
	&\int_0^{a^*}P_m(a)|(f_{1}^{n_k})'(a)|^2\,da+\int_0^{a^*}Q_{n_k}(a)|f_{1}^{n_k}(a)|^2\,da\\
	&\le\int_0^{a^*}P_{n_k}(a)|(f_{1}^{n_k})'(a)|^2\,da+\int_0^{a^*}Q_{n_k}(a)|f_{1}^{n_k}(a)|^2\,da=\kappa^n.
\end{align*}
Passing to the limit as $n_k\to\infty$, we get that
\begin{equation*}
 \int_0^{a^*}P_{m}(a)|(f^*)'(a)|^2\,da+\int_0^{a^*}Q(a)|f^*(a)|^2\,da\le\kappa^*.
\end{equation*}
Since $m$ is arbitrary, we conclude that $f^*$ belongs to $\mathcal F_G$ and that 
\begin{equation*}
	q(f^*)\le \kappa^*.
\end{equation*}
Hence $\kappa^*\ge \kappa_1(G,\beta)$.
\end{proof}

Lemma \ref{lemLimit}, combined with Proposition \ref{propMonotonicity}, yields Proposition \ref{propUpperBound}. Although we do not use it, we can note that the proof of Lemma \ref{lemLimit} shows the existence of a minimizer $f^*\in\mathcal F_G$ for Problem \eqref{eqKappa}, for any $G$ satisfying Assumption \ref{assumG} (not necessarily bounded).

\subsection{Proof of the main theorem}
 
We can now finish the proof of Theorem \ref{thmRFK}. From Propositions \ref{propIsoperimetricG}, \ref{propUpper} and \ref{propUpperBound}, it follows that for any $\beta>0$, 
\[\lambda_1^N(\Omega,\beta)\le \kappa_1(4\pi,\beta),\]
with equality only if $\Omega$ is a disk. On the other hand, as recalled in Section \ref{subsecMagnLap},
\[\kappa_1(4\pi,\beta)=\lambda_1^N(\Omega^*,\beta)\]    whenever $0\le \beta\le \frac{\beta^*\pi}{a^*}$, which proves the theorem.

%%%%%%%%%%%%%%%%%%%%%%%

\appendix

\section{Existence of $\beta^*$}
\label{appBetaStar}

Let us give a proof of Proposition \ref{propBetaStar} from the introduction. We note that Kachmar and Lotoreichik used a similar approach (see the proof of Proposition 3.3 in \cite{KL2022Magnetic}).

\begin{proof}
We have
\[
\kappa_1(n,\beta,1)=\inf_{v\in\mathcal H_n\setminus\{0\}} \mathcal R_n(v)
\]
where
\[
\mathcal R_n(v):=\frac{\int_0^1\left(v'(r)^2+\left(\frac{\beta r}{2}-\frac{n}{r}\right)^2v^2(r)\right)r\,dr}{\int_0^1v^2(r)\,r\,dr},
\]

\[\mathcal H_0:=\{f\in L^2(]0,1[, r\,dr):f'\in L^2(]0,1[, r\,dr)\}\]
and, for $n\neq0$,
\[\mathcal H_n:=\{f\in L^2(]0,1[, r\,dr):f'\in L^2(]0,1[, r\,dr)\,,f/r\in L^2(]0,1[,r\,dr) \}.
\]
We note that $\mathcal H_n\subset\mathcal H_0$ when $n\ne 0$.
Now we see that
\[
\left(\frac{\beta r}{2}-\frac{n}{r}\right)^2=\frac{n^2}{r^2}-n\beta+\frac{\beta^2r^2}{4}>|n|(|n|-\beta)+\frac{\beta^2r^2}{4} \ge\frac{\beta^2r^2}{4}
\]
when $\beta\leq 1$, for all $n\neq0$ and $r\in]0,1[$. This implies that for any $v\in\mathcal H_n$, we have $\mathcal R_n(v)>\mathcal R_0(v)$ for all $n\ne 0$ . 

Applying Rayleigh's principle, we obtain that
\[\kappa_1(n,\beta,1)>\kappa_1(0,\beta,1)\]
for all $\beta\in[0,1]$. This implies the desired result.
\end{proof}

\section{Chain rule formulas}
\label{appChainRules}

Let us first prove Formula \eqref{eqGrad}. We fix $g$, a function in $\mathcal G_\Omega$. We first have to settle a small technical issue concerning the definition of the function $g'\circ\psi$. Since $g'$ is only defined up to equality almost everywhere, we have to check that, given two measurable functions $h_1,h_2:[0,t^*]\to \mathbb C$,  whenever $h_1(t)=h_2(t)$ for almost every $t\in[0,t^*]$, $h_1\circ\psi(p)=h_2\circ\psi(p)$ for almost every $p\in \Omega$. This holds if, and only if, $|E|=0$ implies $|\psi^{-1}(E)|=0$ for any $E\subset[0,t^*]$ measurable. This is equivalent to saying that the push forward of the Lebesgue measure on $\Omega$ by the mapping $\psi$, denoted by ${\psi}_*(dp)$, is absolutely continuous with respect to the Lebesgue measure on $[0,t^*]$. 

The measure ${\psi}_*(dp)$ can easily be computed using the Coarea Formula. Indeed, let us consider $E$, a measurable subset of $[0,t^*]$ and let us apply Remark \ref{remCoarea} to the function $h(p)=\chi_{\psi^{-1}(E)}(p)/|\nabla\psi(p)|$, where $\chi_{\psi^{-1}(E)}$ is the characteristic function of $\psi^{-1}(E)$. We obtain
\begin{align*}
	|\psi^{-1}(E)|&=\int_{\Omega}\chi_{\psi^{-1}(E)}(p)\,dp=\int_{\Omega}h(p)|\nabla \psi(p)|\,dp=\int_0^{t^*}\left(\int_{\Gamma_s}h(q)\,d\,\mathcal H^1(q)\right)\,ds\\
	&=\int_0^{t^*}\left(\int_{\Gamma_s}\frac{{\chi_E(\psi(q))}}{|\nabla\psi(q)|}\,d\,\mathcal H^1(q)\right)\,ds=\int_{0}^{t^*}\chi_E(s)\gamma_\Omega(s)\,ds.
\end{align*}
This shows that ${\psi}_*(dp)=\gamma_\Omega(t)\,dt$, and the right-hand side is absolutely continuous with respect to $dt$, since $\gamma_\Omega$ is integrable.

To show that Formula \eqref{eqGrad} holds in the sense of distribution, we have to show that $p\mapsto g'(\psi(p))\nabla \psi(p)$ is locally integrable and that, for any smooth function $\varphi$ compactly supported in $\Omega$,
\begin{equation} \label{eqDistDer}
	-\int_{\Omega} g(\psi(p))\nabla\varphi(p)\,dp= \int_{\Omega} g'(\psi(p))\varphi(p)\nabla\psi(p)\,dp.
\end{equation}
We first note that, if $g$ is in $C^1(\Omega)$, Equation \eqref{eqDistDer} follows from Green's identity and the standard Chain Rule. 
\begin{lem}\label{lemApproximation}
For any function $g\in\mathcal G$, there exists a sequence $(g_n)$ of functions in $C^1([0,t^*])$ such that 
\begin{enumerate}
	\item $g_n'\to g'$ in $L^2(]0,t^*[,\mu(t)\,
	dt)$;
	\item $g_n \to g$ in $L^2(]0,t^*[,\gamma_\Omega(t)\,
	dt)$.
\end{enumerate}
\end{lem}
\begin{proof} Let us first note that, since $\mu$ in bounded away from $0$ in every compact subset of $[0,t^*[$, $g'$ is in $L^1_{\rm loc}([0,t^*[\,,dt)$, and therefore, for all $t\in [0,t^*[$, 
\[g(t)=g(0)+\int_{0}^{t} g'(s)\,ds.\]
Since $g'\in L^2(]0,t^*[,\mu(t)\,dt)$, there exists a sequence $(h_n)$ of functions which are continuous and compactly supported in $]0,t^*[$, such that $h_n\to g'$ in $L^2(]0,t^*[,\mu(t)\,dt)$. We then define
\[g_n(t):=g(0)+\int_{0}^{t}h_n(s)\,ds\]
for all $n\ge 1$ and $t\in[0,t^*]$. We note that $g_n\in C^1([0,t^*])$. Point 1 is satisfied by construction. Let us now use Fubini's theorem:
\begin{align*} 
	\int_0^{t^*}|g_n(t)-g(t)|^2\gamma_\Omega(t)\,dt&= \int_0^{t^*}\left|\int_0^t (h_n(s)-g'(s))\,ds\right|^2\gamma_\Omega(t)\,dt\\
	&\le t^*\int_0^{t^*}\left(\int_0^t |h_n(s)-g'(s)|^2\,ds\right)\gamma_\Omega(t)\,dt\\
	&= t^* \int_0^{t^*}|h_n(s)-g'(s)|^2\,\left(\int_s^{t^*} \gamma_\Omega(t)\,dt\right)\,ds\\
	&=t^*\int_0^{t^*}|h_n(s)-g'(s)|^2\mu(s)\,ds.
\end{align*}

Since the last term tends to $0$ as $n\to\infty$, we obtain point 2.
\end{proof}
We now fix a sequence $(g_n)$ as in the Lemma. From our initial remark, we have 
\begin{equation} \label{eqDistDern}
	-\int_{\Omega} g_n(\psi(p))\nabla\varphi(p)\,dp= \int_{\Omega} g_n'(\psi(p))\varphi(p)\nabla\psi(p)\,dp
\end{equation}
for all $n$. By starting from points 1 and 2 in the Lemma and applying the Coarea Formula (with computations similar to the proof of Proposition \ref{propH1}) we obtain that, respectively, $g_n\circ\psi$  and $(g_n'\circ\psi)\nabla\psi$ tend to $g\circ\psi$ and $(g'\circ\psi)\nabla\psi$ in $L^2(\Omega)$, and therefore in $L^1(\Omega)$. We then deduce Equation \eqref{eqDistDer} by passing to the limit in Equation \eqref{eqDistDern}.

Finally, let us prove Formula \eqref{eqPrime}. If $f'\in L^1_{\rm loc}(]0,a^*])$, we have, for all $a\in ]0,a^*]$,
\[f(a)=f(a^*)-\int_a^{a^*}f'(\alpha)\,d\alpha\] 
and therefore, from Proposition \ref{propChangeVar}, 
\[g(t)=g(0)-\int_0^tf'(\mu(s))\gamma_{\Omega}(s)\,ds,\]
where $t=\mu^{-1}(a)$. It follows that $g'$ is locally integrable in $[0,t^*[$ and $g'(t)=-f'(\mu(t))\,\gamma_{\Omega}(t)$. 

Conversely, if $g'\in L^1_{\rm loc}([0,t^*[)$, we have, for all $t\in[0,t^*[$, 
\[g(t)=g(0)+\int_0^tg'(t)\,dt,\] 
and therefore, from Proposition \ref{propChangeVar},
\[f(a)=f(a^*)+\int_a^{a^*}\frac{g'(\mu^{-1}(\alpha))}{G(\alpha)}\,d\alpha,\]
where $a=\mu(t)$. It follows that $f'$ is locally integrable in $]0,a^*]$ and $f'(a)=-g'(\mu^{-1}(a))/G(a)$.

\section{Proofs for the Sturm-Liouville problem}
\label{appSL}

Our goal is to prove Propositions \ref{propMinSL} and \ref{propSL}. We follow the well-known method of defining the eigenvalue problem for a self-adjoint operator using a quadratic form. We also prove the properties of the functions $X$ and $Y$ used in the proof of Proposition \ref{propMonotonicity}, more specifically the endpoint behavior stated in Proposition \ref{propSystem}.

Let the function $G:[0,a^*]\to [0,\infty]$ satisfy Assumptions \ref{assumG} and \ref{assumGBounded}. We first prove Lemma \ref{lemFBounded}, which characterizes the associated function space $\mathcal F_G$ (see Definition \ref{dfnG}). Let us recall the statement.

\begin{lem} The set $\mathcal F_G$ is equal to the set $\mathcal F_{4\pi}$ consisting of functions $f:[0,a^*]\to \mathbb C$ satisfying  
\begin{enumerate}
	\item $f\in L^2(]0,a^*[, da)$;
	\item $f'\in L^2(]0,a^*[, 4\pi\,a\,da)$.
\end{enumerate}
In addition, $\mathcal F_{4\pi}$, equipped with the natural norm 
\begin{equation*}
	\|f\|_{\mathcal F_{4\pi}}^2:=\int_0^{a^*}|f(a)|^2\,da+4\pi\int_0^{a^*}|f'(a)|^2a\,da,
\end{equation*} 
is compactly embedded in $L^2(]0,a^*[,da)$.
\end{lem}

\begin{proof} From Assumptions \ref{assumG} and \ref{assumGBounded}, we have 
\[4\pi\le G(a)\le \|G\|_{L^\infty}\]
for almost every $a$ in $[0,a^*]$, which immediately implies that the sets $\mathcal F_G$ and $\mathcal F_{4\pi}$ are equal.

Among several possible approaches to prove the compact embedding, we use one which is closely related to the method of level lines. We define $R$ by $\pi R^2=a^*$, so that $B_R$ is the disk of area $a^*$. We map any function $f\in\mathcal F_{4\pi}$ to a function $u$ on $B_R$, defined by
\begin{equation}
	u(p)=f(\pi|p|^2)
\end{equation}
for $p\in B_R$.

Integrating in polar coordinates over $B_R$ and using Fubini's theorem, we find
\begin{equation*}
	\int_{B_R}|u(p)|^2\,dp=\int_0^{2\pi}\left(\int_{0}^R|f(\pi r^2)|^2r\,dr\right)\,d\theta=2\pi\int_{0}^R|f(\pi r^2)|^2r\,dr.
\end{equation*}
The change of variable $a=\pi r^2$ in the last integral gives
\begin{equation*}
	2\pi\int_{0}^R|f(\pi r^2)|^2r\,dr=\int_0^{a^*}|f(a)|^2\,da.
\end{equation*}

In addition, a computation similar to the proof of Proposition  \ref{propH1} in Appendix \ref{appChainRules} shows that the gradient of $u$, in the sense of distribution, is
\begin{equation*}
	\nabla u(p)= 2\pi f'(\pi|p|^2)p.
\end{equation*}
Integrating as before, we find
\begin{equation*}
	\int_{B_R}|\nabla u(p)|^2\,dp=8\pi^3\int_{0}^R|f'(\pi r^2)|^2r^3\,dr=4\pi\int_0^{a^*}|f'(a)|^2a\,da.
\end{equation*}
From this and the previous paragraph, we conclude that the mapping $f\mapsto u$ defined above sends $\mathcal F_{4\pi}$ into $H^1(B_R)$, isometrically if $\mathcal F_{4\pi}$ is equipped with the natural norm $\|\cdot\|_{\mathcal F_{4\pi}}$, and that it also separately preserves the $L^2$-norm. 

The conclusion is now easy to reach. Let us assume that $(f_n)$ is a bounded sequence in $(\mathcal F_{4\pi},\|\cdot\|_{\mathcal F_{4\pi}})$. It is mapped to a bounded sequence $(u_n)$ in $H^1(B_R)$. By Rellich's theorem, $H^1(B_R)$ is compactly embedded in $L^2(B_R)$, and therefore $(u_{n})$ admits a subsequence $(u_{n_k})$ which converges in $L^2(B_R)$. By the isometric mapping, the corresponding subsequence $(f_{n_k})$ is Cauchy, and therefore convergent, in $L^2(]0,a^*[,da)$.
\end{proof}

\begin{rem} Let us note that $\pi|p|^2$ is the area enclosed by the circle centered at $0$ passing through the point $p$ in $B_R$. Since the level lines of the torsion function in $B_R$ are circles, we have $\pi|p|^2=\mu(\psi(p))$ for all $p\in B_R$. Here we are using the notation of Section \ref{secArea} in the case $\Omega=B_R$ and $\psi(p)=\frac14(R^2-|p|^2)$. The mapping $f(a) \mapsto u(p)$ is therefore the one defined in Corollary \ref{corH1}.
\end{rem}  

We now consider the quadratic form defined, for $f\in\mathcal F_{4\pi}$, by
\begin{equation*}
	q(f)=\int_0^{a^*}\left(P(a)|f'(a)|^2+\beta^2 Q(a)| f(a)|^2\right)\,da.
\end{equation*}
Since $|P(a)|\le a\,\|G\|_{L^\infty}$ and $|Q(a)|\le \frac{a}{4\pi}$  for almost every $a$, the form $q$ is continuous in $(\mathcal F_{4\pi},\|\cdot\|_{\mathcal F_{4\pi}})$. Since $Q$ is non-negative, $q$ is non-negative and in particular bounded from below. Standard results from spectral theory then tell us that there exists a sequence $(\kappa_k,f_k)_{k\ge1}$ of spectral pairs satisfying the following properties (where $\langle\cdot,\cdot\rangle$ denotes the scalar product in $L^2(]0,a^*[,da)$ and $q(\cdot,\cdot)$ the sesquilinear form associated with the quadratic form $q(\cdot)$).
\begin{enumerate}
	\item For all $k$, $f_k\in\mathcal F_{4\pi}$ and $f_k$ is an eigenfunction of $q$ associated with the  eigenvalue $\kappa_k$, that is to say
	\begin{equation}
		\label{eqWeak}
		q(f_k,\varphi)=\kappa_k\langle f_k,\varphi\rangle
	\end{equation}
	for all $\varphi\in\mathcal F_{4\pi}$.
	\item The sequence $(f_k)$ is a Hilbert basis of $L^2(]0,a^*[,da)$.
	\item The sequence $(\kappa_k)$ tends to $+\infty$.
\end{enumerate}

Setting $\kappa_k(G,\beta):=\kappa_k$, we see that we have proved point 1 in Proposition \ref{propSL}.

In addition, as a consequence of the variational characterization of eigenvalues, we have 
\begin{equation}
	\label{eqVariational}
	\kappa_1=\min_{f\in\mathcal F_{4\pi}\setminus\{0\}}\frac{q(f)}{\|f\|_{L^2}^2}.
\end{equation}
This proves Proposition \ref{propMinSL}.

The fact that $q$  is associated to a Sturm-Liouville operator implies further properties of the eigenvalues and eigenfunctions. To see this, let us fix $k$ and let us consider the weak eigenvalue equation \eqref{eqWeak}. Since $\mathcal F_{4\pi}$ contains, in particular, all smooth functions with compact support in $]0,a^*[$, Equation \eqref{eqWeak} implies that 
\begin{equation}
\label{eqEV}
	-\left(P(a)f_k'\right)'+\beta^2Q(a)f_k=\kappa_ k f_k
\end{equation}
in the sense of distributions in $]0,a^*[$. Introducing the quasi-derivative $f_k^{[1]}(a):=P(a)f_k'(a)$ and rearranging the terms, we obtain
\begin{equation*}
		(f_k^{[1]})'(a)=\left(\beta^2Q(a)-\kappa_k \right)f_k(a)
\end{equation*}
in the sense of distribution. The left-hand side is in $L^2$, and thus in $L^1$, meaning that $f_k^{[1]}$ is absolutely continuous in $[0,a^*]$. Since 
\begin{equation*}
	f_k'(a)=\frac1{P(a)}f_k^{[1]}(a)
\end{equation*}
and
\begin{equation*}
			0\le\frac1{P(a)}\le\frac1{4\pi\,a},
\end{equation*}
it follows that $f_k$ is locally absolutely continuous in $]0,a^*]$. We have thus proved point 3 of Proposition \ref{propSL}.

Let us now consider a function $\varphi$ in $C^1([0,a^*])$ . According to Equation \eqref{eqEV},  we have, 
\begin{equation*}
	\kappa_k\int_0^{a^*}f_k(a)\varphi(a)\,da=\int_0 \left(-(f_k^{[1]}(a))'+\beta^2Q(a)f_k(a)\right)\varphi(a)\,da.
\end{equation*}
Applying the integration by part formula for absolutely continuous functions, we obtain
\begin{equation*}
	\kappa_k\int_0^{a^*}f_k(a)\varphi(a)\,da=[f_k^{[1]}(a)\varphi(a)]_0^{a^*}+\int_0^{a^*}\left(f_k^{[1]}(a)\varphi'(a)+\beta^2Q(a)f_k(a)\varphi(a)\right)\,da.
\end{equation*} 
Since, according to Equation \eqref{eqWeak},
\begin{equation*}
	\int_0^{a^*}\left(f_k^{[1]}(a)\varphi'(a)+\beta^2Q(a)f_k(a)\varphi(a)\right)\,da=\kappa_k\int_0^{a*}f_k(a)\varphi(a)\,da,
\end{equation*}
we have shown that 
\begin{equation*}
f_k^{[1]}(a^*)\varphi(a^*)-f_k^{[1]}(0)\varphi(0)=0
\end{equation*}
for all $\varphi\in C^1([0,a^*])$. By applying this successively with $\varphi_1$ such that $\varphi_1(0)=0$ and $\varphi_1(a^*)=1$ and with $\varphi_2$ such that $\varphi_2(0)=1$ and $\varphi_2(a^*)=0$, we obtain 
\begin{equation}\label{eqBC}
	f_k^{[1]}(a^*)=f_k^{[1]}(0)=0.
\end{equation}
We have therefore proved the boundary conditions in Problem \eqref{eqSL}. 

Re-using the notation of Section \ref{secSL}, we set
\begin{align*}
	X(a):=&f_k(a),\\
	Y(a):=&f_k^{[1]}(a).
\end{align*}
Then, the vector-valued function $a\mapsto (X(a),Y(a))$ satisfies the system \eqref{eqSystemODE} with $\kappa=\kappa_n$. Let us note that the function $a\mapsto \beta^2Q(a)-\kappa_n$ is in $L^\infty(]0,a^*[)$ and that the function $a\mapsto 1/P(a)=(a\,G(a))^{-1}$ is in $L^{\infty}_{\rm loc}(]0,a^*])$. The system \eqref{eqSystemODE} therefore satisfies  in $]0,a^*]$ the existence and uniqueness property for the solution to a Cauchy problem. As a consequence, as in the Sturm-Liouville theory with regular coefficients, we have the following results.
\begin{enumerate}
	\item Any eigenfunction associated with $\kappa_k$ is proportional to $f_k$, that is, the eigenvalue $\kappa_k$ is simple.
	\item $f_k(a^*)\neq 0$.
	\item  If $f_k$ vanishes at $a_0\in]0,a^*[$, then it changes sign. In particular, every interior zero of $f_k$ is isolated.  
\end{enumerate}
Let us note that the first point above is point 2 in Proposition \ref{propSL}. 

Let us now prove that $f_1$ does not change sign . Since $f_1$  is absolutely continuous, so is $|f_1|$, with $||f_1|'(a)|=|f_1'(a)|$ almost everywhere. We therefore have
\begin{equation*}
	\frac{q\left(|f_1|\right)}{\|\,|f_1|\,\|_{L^2}^2}=\frac{q(f_1)}{\|f_1\|_{L^2}^2}.
\end{equation*}
From the variational characterization of eigenvalues, it follows that $|f_1|$ is an eigenfunction associated with $\kappa_1$. Therefore, $f_1$ is proportional to $|f_1|$ and does not change sign in $]0,a^*[$. According to point 3 above, this implies that $f_1$ does not vanish in $]0,a^*[$.   From now on, we chose $f_1$ strictly positive in $]0,a^*[$. Let us note that at this stage, we have proved the endpoint conditions $Y(0)=Y(a^*)=0$ and $X(a^*)>0$ stated in Proposition \ref{propSystem}.

The existence and uniqueness theorem may not apply at $a=0$, since $a\mapsto 1/P(a)$ may not be integrable there. However, when $n=1$, we can prove that $X=f_1$ has a finite limit as $a\to 0$ (denoted by $X(0)$) and that $X(0)>0$, as stated in Proposition \ref{propSystem}.  Indeed, $X(a)>0$ and, as $a\to 0$, 
\begin{equation*}
	\beta^2Q(a)-\kappa_1=-\kappa_1+O(a),
\end{equation*}
so that $\beta^2Q(a)-\kappa_1<0$ for $a$ close to $0$. Using the second equation in System \eqref{eqSystemODE}, it follows that $Y'(a)<0$ for $a$ close to $0$. The function $a\mapsto X(a)$ is therefore decreasing in a neighborhood of $0$, and thus, as $a\to 0$,
\begin{equation*}
	X(a)\to \ell,
\end{equation*}
with $\ell > 0$ finite or infinite. Therefore $\ell>0$. According to the first differential equation in System \eqref{eqSystemODE}, for almost every $a$,
\begin{equation*}
	|X'(a)|\le \frac{\|Y\|_{\infty}}{4\pi\, a},
\end{equation*}
 and it follow by integration that there exists positive constants $C_1$ and $C_2$ such that 
\begin{equation*}
	|X(a)|\le C_1|\log(a)|+C_2.
\end{equation*}
According to the second differential equation,
\begin{equation}
\label{eqBoundY}
	Y(a)=\int_0^a (Q(\alpha)-\kappa)X(\alpha)\,d\alpha= \int_{0}^a\left(\frac{\alpha}{G(\alpha)}-\kappa\right)X(\alpha)\,d\alpha
\end{equation}
and therefore, for $a$ in some interval $]0,\varepsilon[$,
\begin{equation*}
	|Y(a)|	\le \int_{0}^a\left(\frac{\alpha}{G(\alpha)}+\kappa\right)\left(C_1|\log(\alpha)|+C_2\right)\,d\alpha\le C_3\,a\,|\log(a)|,
\end{equation*}
with $C_3$ a positive constant. Hence, for almost every $a\in ]0,\varepsilon[$,
\begin{equation*}
	|X'(a)|\le \frac{C_3}{4\pi}|\log(a)|.
\end{equation*}
It follows by integration that $X$ is bounded, and therefore that $\ell$ is finite, as required. Using again the second differential equation in system \eqref{eqSystemODE}, we find that  $Y(a)=O\left(a\right)$ as $a\to 0$. This completes the proof of Proposition \ref{propSystem}.

\bibliographystyle{plain}

\end{document}